\documentclass[12pt]{amsart}

\usepackage[matrix,arrow,curve,frame]{xy}
\usepackage{amsmath,amsthm,amssymb,enumerate}
\usepackage{latexsym}
\usepackage{amscd}
\usepackage[colorlinks=true]{hyperref}
\usepackage{euscript}

\newtheorem{thm}[subsection]{Theorem}
\newtheorem{defn}[subsection]{Definition}

\newtheorem{claim}[subsection]{Claim}

\newtheorem{lemma}[subsection]{Lemma}

\newtheorem{remark}[subsection]{Remark}

\theoremstyle{definition}

\newcommand{\cat}{\mathcal}
\newcommand{\lra}{\longrightarrow}

\newcommand{\llra}[1]{\stackrel{#1}{\lra}}

\newcommand{\R}{\mathbb R}

\newcommand{\Z}{\mathbb Z}

\newcommand{\C}{\mathbb C}

\newcommand{\Kth}{\mathbb K}
\newcommand{\Kk}{\mathbb K\mathbb U}
\newcommand{\Dirac}{\textup{\mbox{D\hspace{-0.6em}\raisebox{.1ex}%
{/}\hspace{.07em}}}}

\newcommand{\gG}{\mbox{\german g}}

\newcommand{\gH}{\mbox{\german h}}
\newcommand{\gB}{\mbox{\german b}}
\newcommand{\gZ}{\mbox{\german Z}}

\newcommand{\gC}{{\cat S}}
\newcommand{\G}{{\cat G}}

\DeclareMathOperator{\hocolim}{hocolim}

\DeclareMathOperator{\Map}{Map}

\DeclareMathOperator{\Hom}{Hom}

\newfont{\german}{eufm10}

\begin{document}
\pagestyle{plain}

\title
{Dominant K-theory and Integrable highest weight representations of Kac-Moody groups}
\author{Nitu Kitchloo}
\address{Department of Mathematics, University of California, San
  Diego, USA}
\email{nitu@math.ucsd.edu}
\thanks{Nitu Kitchloo is supported in part by NSF through grant DMS
  0705159.}

\date{\today}


\dedicatory{To Haynes Miller on his 60th birthday}

{\abstract

\noindent
We give a topological interpretation of the highest weight representations of Kac-Moody groups. Given the unitary form $\G$ of a Kac-Moody group (over $\C$), we define a version of equivariant K-theory, $\Kth_{\G}$ on the category of proper $\G$-CW complexes. We then study Kac-Moody groups of compact type in detail (see Section 2). In particular, we show that the Grothendieck group of integrable hightest weight representations of a Kac-Moody group $\G$ of compact type, maps isomorphically onto $\tilde{\Kth}_{\G}^*(E\G)$, where $E\G$ is the classifying space of proper $\G$-actions. In the Affine case, this agrees very well with recent results of Freed-Hopkins-Teleman. We also explicitly compute $\Kth_{\G}^*(E\G)$ for Kac-Moody groups of extended compact type, which includes the Kac-Moody group $\mbox{E}_{10}$.}
\maketitle

\tableofcontents
\section{Introduction}

\noindent
Given a compact Lie group $G$ and a $G$-space $X$, the equivariant K-theory of $X$, $K_{G}^*(X)$ can be described geometrically in terms of equivariant vector bundles on $X$. If one tries to relax the condition of $G$ being compact, one immediately runs into technical problems in the definition of equivariant K-theory. One way around this problem is to impose conditions on the action of the group $G$ on the space $X$. By a proper action of a topological group $\G$ on a space $X$, we shall mean that $X$ has the structure of a $\G$-CW complex with compact isotropy subgroups. There exists a universal space with a proper $\G$-action, known as the classifying space for proper $\G$-actions, which is a terminal object (up to $\G$-equivariant homotopy) in the category of proper $\G$-spaces. One may give an alternate identification of this classifying space as:
\begin{defn} For a topological group $\G$, the  classifying space for proper actions $E\G$ is a $\G$-CW complex with the property that all the isotropy subgroups are compact, and given a compact subgroup $H \subseteq \G$, the fixed point space $E\G^H$ is contractible. 
\end{defn}

\noindent
Notice that if $\G$ is a compact Lie group, then $E\G$ is simply equivalent to a point, and so $K_{\G}^*(E\G)$ is isomorphic to the representation ring of $\G$. For a general non-compact group $\G$, the definition or geometric meaning of $K_{\G}^*(E\G)$ remains unclear. 

\noindent
In this paper, we deal with a class of topological groups known as Kac-Moody groups \cite{K1,K2}. By a Kac-Moody group, we shall mean the unitary form of a split Kac-Moody group over $\C$. We refer the reader to \cite{Ku} for a beautiful treatment of the subject. These groups form a natural extension of the class of compact Lie groups, and share many of their properties. They are known to contain the class of (polynomial) Loop groups, which go by the name of Affine Kac-Moody groups. With the exception of compact Lie groups, Kac-Moody groups over $\C$ are not even locally compact (local compactness holds for Kac-Moody groups defined over finite fields). However, the theory of integrable highest weight representations does extend to the world of Kac-Moody groups. These representations are none other than the positive energy representations in the case of a Loop group. It is therefore natural to ask if some version of equivariant K-theory for Kac-Moody groups encodes the integrable highest weight representations.

\noindent
The object of this paper is two fold. Firstly, we define equivariant K-theory $\Kth_{\G}$ as a functor on the category of proper $\G$-CW complexes, where $\G$ is a Kac-Moody group. For reasons that will become clear later, we shall call this functor "Dominant K-theory". For a compact Lie group $\G$, this functor is usual equivariant K-theory. Next, we build an explicit model for the classifying space of proper $\G$-actions, $E\G$, and calculate the groups $\Kth_{\G}^*(E\G)$ for Kac-Moody groups $\G$ of compact, and extended compact type (see Section 2). Indeed, we construct a map from the Grothendieck group of highest weight representations of $\G$, to $\Kth_{\G}^*(E\G)$ for a Kac-Moody group of compact type, and show that this map is an isomorphism. 

\noindent
This document is directly inspired by the following recent result of Freed, Hopkins and Teleman \cite{FHT}: Let $\mbox{G}$ denote a compact Lie group and let $\tilde{\mbox{L}}\mbox{G}$ denote the universal central extension of the group of smooth free loops on $\mbox{G}$. In \cite{FHT}, the authors calculate the twisted equivariant K-theory of the conjugation action of $\mbox{G}$ on itself and describe it as the Grothendieck group of positive energy representations of $\tilde{\mbox{L}}\mbox{G}$. The positive energy representations form an important class of (infinite dimensional) representations of $\tilde{\mbox{L}}\mbox{G}$ \cite{PS} that are indexed by an integer known as 'level'. If $\mbox{G}$ is simply connected, then the loop group $\tilde{\mbox{L}}\mbox{G}$ admits a Kac-Moody form known as the Affine group. We will show that the Dominant K-theory of the classifying space of proper actions of the corresponding Affine Kac-Moody group is simply the graded sum (under the action of the central circle) of the twisted equivariant K-theory groups. Hence, we recover the theorem of Freed-Hopkins-Teleman in the special case of the Loop group of a simply connected, simple, compact Lie group. 

\subsection{Organization of the Paper:}

\noindent

\noindent
We begin in Section 2 with background on Kac-Moody groups and describe the main definitions, constructions and results of this paper. Section 3 describes the finite type topological Tits building as a model for the classifying space for proper actions of a Kac-Moody group, and in Section 4 we study the properties of Dominant K-theory as an equivariant cohomology theory. In Sections 5 we compute the Dominant K-theory of the Tits building for a Kac-Moody group of compact type, and in section 6 we give our computation a geometric interpretation in terms of an equivariant family of cubic Dirac operators. Section 7 explores the Dominant K-theory of the building for the extended compact type. Section 8 is a discussion on various related topics, including the relationship of our work with the work of Freed-Hopkins-Teleman. Included also in this section are remarks concerning real forms of Kac-Moody groups, and the group $\mbox{E}_{10}$. Finally, in section 9, we define a corresponding Dominant K-homology theory, by introducing the equivariant dual of proper complexes. We also compute the Dominant K-homology of the building in the compact type. The Appendix is devoted to the construction of a compatible family of metrics on the Tits Building, which is required in section 6.

\subsection{Acknowledgements}

\noindent

\noindent
The author acknowledges his debt to the ideas described in \cite{FHT} and \cite{AS}. He would also like to thank P. E. Caprace, L. Carbone, John Greenlees and N. Wallach for helpful feedback. 

\section{Background and Statement of Results}

\noindent
Kac-Moody groups have been around for about thirty years \cite{K1,K2} (see \cite{Ku} for a modern prespective). One begins with a finite integral matrix $A = (a_{i,j})_{i,j \in I}$ with the properties that $a_{i,i} = 2$ and $a_{i,j} \leq 0$ for $i\neq j$. Moreover, we demand that $a_{i,j} = 0$ if and only if $a_{j,i} = 0$. These conditions define a generalized Cartan Matrix. A generalized Cartan matrix is said to be symmetrizable if it becomes symmetric after multiplication with a suitable rational diagonal matrix. 

\noindent
Given a generalized Cartan matrix $A$, one may construct a complex Lie algebra $\gG (A)$ using the Harishchandra-Serre relations. This Lie algebra contains a finite dimensional Cartan subalgebra $\gH$ that admits an integral form $\gH_{\Z}$ and a real form $\gH_{\R} = \gH_{\Z}\otimes \R$. The lattice $\gH_{\Z}$ contains a finite set of primitive elements $h_i, i \in I$ called "simple coroots".  Similarly, the dual lattice $\gH_{\Z}^*$ contains a special set of elements called "simple roots" $\alpha_i, i \in I$. One may decompose $\gG(A)$ under the adjoint action of $\gH$ to obtain a triangular form as in the classical theory of semisimple Lie algebras. Let $\eta_{\pm}$ denote the positive and negative "nilpotent" subalgebras respectively, and let $\gB_{\pm} = \gH \oplus \eta_{\pm}$ denote the corresponding "Borel" subalgebas. The structure theory for the highest weight representations of $\gG(A)$ leads to a construction (in much the same way that Chevalley groups are constructed), of a topological group $G(A)$ called the (minimal, split) Kac-Moody group over the complex numbers. The group $G(A)$ supports a canonical anti-linear involution $\omega$, and one defines the unitary form $K(A)$ as the fixed group $G(A)^{\omega}$. It is the group $K(A)$ that we study in this article. We refer the reader to \cite{Ku} for details on the subject. 

\noindent
Given a subset $J \subseteq I$, one may define a parabolic subalgebra $\gG_J(A) \subseteq \gG(A)$ generated by $\gB_+$ and the root spaces corresponding to the set $J$. For example, $\gG_{\varnothing} (A)= \gB_+$. One may exponentiate these subalgebras to parabolic subgroups $G_J(A) \subset G(A)$. We then define the unitary Levi factors $K_J(A)$ to be the groups $K(A) \cap G_J(A)$. Hence $K_{\varnothing} (A) = T$ is a torus of rank $2|I| - rk(A)$, called the maximal torus of $K(A)$. The normalizer $N(T)$ of $T$ in $K(A)$, is an extension of a discrete group $W(A)$ by $T$. The Weyl group $W(A)$ has the structure of a crystallographic Coxeter group generated by reflections $r_i, i \in I$. For $J \subseteq I$, let $W_J(A)$ denote the subgroup generated by the corresponding reflections $r_j, j \in J$. The group $W_J(A)$ is a crystallographic Coxeter group in its own right that can be identified with the Weyl group of $K_J(A)$. 

\noindent
Let Aut($A$) denote the automorphism group of $K(A)$ induced from automorphisms of $\gG(A)$. The outer automorphism group of $K(A)$ is essentially a finite dimensional torus, extended by a finite group of automorphisms of the Dynkin diagram of $A$ \cite{KW}. Given a torus $T' \subseteq \mbox{Aut}(A)$, the reader may verify that all constructions and theorems in this document have obvious analogs for groups of the form $T' \ltimes K(A)$. 

\noindent
We will identify the type of a Kac-Moody group $K(A)$, by that of its generalized Cartan matrix. For example, a generalized Cartan matrix $A$ is called of {\em Finite type} if the Lie algebra $\gG(A)$ is a finite dimensional semisimple Lie algebra. In this case, the groups $G(A)$ are the corresponding simply connected semisimple complex Lie groups. Another sub-class of groups $G(A)$ correspond to Cartan matrices that are of  {\em Affine Type}. These are non-finite type matrices $A$ which are diagonalizable with nonnegative real eigenvalues. The group of polynomial loops on a complex simply-connected semisimple Lie group can be seen as Kac-Moody group of Affine type. Kac-Moody groups that are not of finite type or of affine type, are said to be of {\em Indefinite type}. We will say that a generalized Cartan matrix $A$ is of {\em Compact Type}, if for every proper subset $J \subset I$, the sub matrix $(a_{i,j})_{i,j \in J}$ is of Finite type (see \cite{D3} (Section 6.9) for a classification). It is known that indecomposable generalized Cartan matrices of Affine type are automatically of Compact type. Finally, for the purposes of this paper, we introduce the {\em Extended Compact Type} defined as a generalized Cartan matrix $A = (a_{i,j})_{i,j \in I}$, for which there is a (unique) decompositon $I = I_0 \coprod J_0$, with the property that the sub Cartan matrix $(a_{i,j})_{i,j \in J}$ is of non-finite type if and only if $I_0 \subseteq J$.

\noindent
{\em To avoid redundancy, we will prove our results under the assumption that the Kac-Moody group under consideration is not of finite type.}

\noindent
Given a generalized Cartan matrix $A = (a_{i,j})_{i,j \in I}$, define a category $\gC(A)$ to be the poset category (under inclusion) of subsets $J \subseteq I$ such that $K_J(A)$ is a compact Lie group. This is equivalent to demanding that $W_J(A)$ is a finite group. Notice that $\gC(A)$ contains all subsets of $I$ of cardinality less than two. In particular, $\gC(A)$ is nonempty and has an initial object given by the empty set. However, $\gC(A)$ does not have a terminal object since $A$ is assumed to not be of finite type. The category $\gC(A)$ is also known as the poset of spherical subsets \cite{D4}. 

\begin{remark}
The topology on the group $K(A)$ is the strong topology generated by the compact subgroups $K_J(A)$ for $J \in \gC(A)$ \cite{K2,Ku}. More precisely, $K(A)$ is the amalgamated product of the compact Lie groups $K_J(A)$, in the category of topological groups. For a arbitrary subset $L \subseteq I$, the topology induced on homogeneous space of the form $K(A)/K_L(A)$ makes it into a CW-complex, with only even cells, indexed by the set of cosets $W(A)/W_L(A)$. 
\end{remark}

\noindent
We now introduce the topological Tits building, which is a space on which most of our constructions rest. Assume that the set $I$ has cardinality $n+1$, and let us fix an ordering of the elements of $I$. Notice that the geometric $n$ simplex $\Delta(n)$ has faces that can be canonically identified with proper subsets of $I$. Hence, the faces of codimension $k$ correspond to subset of cardinality $k$. Let $\Delta_J(n)$ be the face of $\Delta(n)$ corresponding to the subset $J$, and let $\mathcal{B}\Delta(n)$ denote the Barycentric subdivision of $\Delta(n)$. It follows that the faces of dimension $k$ in $\mathcal{B}\Delta(n)$ are indexed on chains of length $k$ consisting of proper inclusions $\varnothing \subseteq J_1 \subset J_2 \subset \ldots J_k \subset I$.  Let $|\gC(A)|$ denote the subcomplex of $\mathcal{B}\Delta(n)$ consisting of those faces for which the corresponding chain is contained entirely in $\gC(A)$. Henceforth, we identify $|\gC(A)|$ as a subspace of $\Delta(n)$. The terminology is suggestive of the fact that $|\gC(A)|$ is canonically homeomorphic to the geometric realization of the nerve of the category $\gC(A)$.

\begin{defn} Define the (finite-type) Topological Tits building $X(A)$ as the $K(A)$-space:
\[ X(A) = \frac{K(A)/T \times |\gC(A)|}{\sim},   \]
where we identify $(gT,x)$ with $(hT,y)$ iff $x=y \in \Delta_J(n)$, and $g=h \, \mbox{mod} \, K_J(A)$. 
\end{defn}
\begin{remark}
An alternate definition of $X(A)$ is as the homotopy colimit \cite{BK}:
\[ X(A) = {\rm hocolim}_{J \in \gC(A)} \; F(J),  \]
where $F$ is the functor from $\gC(A)$ to $K(A)$-spaces, such that $F(J) = K(A)/K_J(A)$. 
\end{remark}
\noindent 
Notice that by construction, $X(A)$ is a $K(A)$-CW complex such that all the isotropy subgroups are compact Lie groups. In fact, we prove the following theorem in Section 3.

\begin{thm} The space $X(A)$ is equivalent to the classifying space $EK(A)$ for proper $K(A)$-actions. 
\end{thm} 
\noindent
We therefore have equivariant maps $\iota_J : K(A)/K_J(A) \rightarrow X(A)$ for any $J \in \gC(A)$. Note that by the above theorem, these maps are unique up to equivariant homotopy. 

\medskip
\noindent
The Weyl chamber C, and the space Y, are defined as subspaces of $\gH_{\R}^*$: 
\[ \mbox{C} = \{ \lambda \in \gH_{\R}^*\, | \, \lambda(h_i)\geq 0, \, i \in I \}, \quad \quad \mbox{Y} = W(A)\, \mbox{C}. \]
Define $\mbox{C}_J$ as the face of $\mbox{C}$ given by:
\[ \mbox{C}_J = \mbox{Interior} \, \{ \lambda \in \mbox{C} | \, \lambda(h_j)=0, \forall j \in J \}. \]
The space Y has the structure of a cone. Indeed it is called the Tits cone. The Weyl chamber is the fundamental domain for the $W(A)$-action on the Tits cone. Moreover, the stabilizer of any point in $\mbox{C}_J$ is $W_J(A)$.

\medskip
\noindent
The subset of $\gH_{\Z}^*$ contained in C is called the set of dominant weights D, and those in the interior of C are called the regular dominant weights $\mbox{D}_+$:
\[ \mbox{D} = \{ \lambda \in \gH_{\Z}^* \; | \; \lambda(h_i) \geq 0, \; i \in I \} \quad \quad  \mbox{D}_+ = \{ \lambda \in \gH_{\Z}^* \; | \; \lambda(h_i) > 0, \; i \in I \}. \]
\noindent
As is the case classically, there is a bijective correspondance between the set of irreducible highest weight representations, and the set D. This bijection identifies an irreducible representation $\mbox{L}_{\mu}$ of hightest weight $\mu$, with the element $\mu \in \mbox{D}$. It is more convenient for our purposes to introduce a {\em Weyl element} $\rho \in \gH_{\Z}^*$ with the property that $\rho(h_i)=1$ for all $i\in I$. The Weyl element may not be unique, but we fix a choice. We may then index the set of irreducible highest weight representations on $\mbox{D}_+$ by identifying an irreducible representation $\mbox{L}_{\mu}$ of hightest weight $\mu$, with the element $\mu + \rho \in \mbox{D}_+$. The relevance of the element $\rho$ will be made clear in Section 5. 

\noindent
For the sake of simplicity, assume that $A$ is a symmetrizable generalized Cartan matrix. In this case, any irreducible representation $\mbox{L}_{\mu}$ admits a $K(A)$-invariant Hermitian inner product. Let us now define a dominant representation of $K(A)$ in a Hilbert space:
\begin{defn}
We say that a $K(A)$-representation in a seperable Hilbert space is dominant if it decomposes as a sum of irreducible highest weight representations. In particular, we obtain a maximal dominant $K(A)$-representation $\mathcal{H}$ (in the sense that any other dominant representation is a summand) by taking a completed sum of (countably many) copies of all irreducible highest weight representations. \end{defn}

\medskip
\noindent
{ \bf The Dominant K-theory spectrum}:

\noindent

\noindent
We are now ready to define Dominant K-theory as a cohomology theory on the category of proper $K(A)$-CW complexes. As is standard, we will do so by defining a representing object for Dominant K-theory. Recall that usual 2-periodic K-theory is represented by homotopy classes of maps into the infinite grassmannian $\Z \times \mbox{BU}$ in even parity, and into the infinite unitary group $\mbox{U}$ in odd parity. The theorem of Bott periodicity relates these spaces via $\Omega \mbox{U} = \Z \times \mbox{BU}$, ensuring that this defines a cohomology theory. 

\noindent
This structure described above can be formalized using the notion of a spectrum. In particular, a spectrum consists of a family of pointed spaces $\mathbb{E}_n$ indexed over the integers, endowed with homeomorphisms $\mathbb{E}_{n-1} \rightarrow \Omega \mathbb{E}_n$. The objects that represent an equivariant cohomology theory are known as equivariant spectra. In analogy to a spectrum, an equivariant spectrum $\mathbb{E}$ consists of a collection of pointed $K(A)$-spaces $\mathbb{E}(V)$, indexed on finite dimensional sub-representations $V$ of an infinite dimensional unitary representation of $K(A)$ in a separable Hilbert space (known as a "universe"). In addition, these spaces are related on taking suitable loop spaces. For a comprehensive reference on equivariant spectra, see  \cite{LMS}. 

\noindent
Henceforth, all our $K(A)$-equivariant spectra are to be understood as naive equivariant spectra, i.e. indexed on a trivial $K(A)$-universe. Indeed, there are no interesting finite dimensional representations of $K(A)$, and so it is even unclear what a nontrivial universe would mean. 

\medskip
\noindent
Let $\Kk$ denote the $K(A)$-equivariant periodic K-theory spectrum represented by a suitable model $\mathcal{F}(\mathcal{H})$, for the space of Fredholm operators on $\mathcal{H}$ \cite{AS,S}. The space $\mathcal{F}(\mathcal{H})$ is chosen so that the projective unitary group $\mathbb{PU}(\mathcal{H})$ (with the compact open topology) acts continuously on $\mathcal{F}(\mathcal{H})$ \cite{AS} (Prop 3.1). By maximality, notice that we have an equivalence: $\mathcal{H}\otimes \mathcal{H} = \mathcal{H}$. Hence $\Kk$ is naturally an equivariant ring spectrum.

\medskip
\noindent
For a proper $K(A)$-CW spectrum $Y$, we define Dominant K-cohomology as:
\begin{defn} 
The Dominant K-cohomology groups of a proper $K(A)$-CW spectrum $Y$ are defined as the group of equivariant homotopy classes of stable maps:
\[ \Kth_{K(A)}^{k}(Y) = [Y, \Sigma^k \Kk]_{K(A)}.  \]

\noindent
Given a closed subgroup $G \subseteq K(A)$, we introduce the notation ${}^A \Kth^*_{G}(X)$, when necessary, to mean the restriction of Dominant K-theory from $K(A)$-spectra to $G$-spectra. Hence
\[ {}^A \Kth^k_G(X) = \Kth_{K(A)}^k(K(A)_+ \wedge_{G} X) = [X, \Sigma^k \Kk]_G.  \]
\end{defn}
\begin{remark} The reader interested only in spaces may replace the definition of Dominant K-theory as follows: Given a proper $K(A)$-CW complex $Y$, Dominant K-theory can be defined as the group of equivariant homotopy classes of maps:
\[ \Kth_{K(A)}^{2k}(Y) = [Y, \mathcal{F}(\mathcal{H})]_{K(A)}, \quad \mbox{and} \quad \Kth_{K(A)}^{2k+1}(Y) = [Y, \Omega \mathcal{F}(\mathcal{H})]_{K(A)}.  \]
\end{remark}
\noindent
In section 5 we plan to prove the following theorem:
\begin{thm} 
Assume that the generalized Cartan matrix $A$ has size $n+1$, and that $K(A)$ is of compact type, but not of finite type. Let $T \subset K(A)$ be the maximal torus, and let $\mbox{R}_T^{\varnothing}$ denote the ideal within the character ring of $T$ generated by the regular dominant weights, $\mbox{D}_+$. 

\noindent
For the space $X(A)$, define the reduced Dominant K-theory $\tilde{\Kth}_{K(A)}^*(X(A))$, to be the kernel of the restriction map along any $T$-orbit in $X(A)$. Then we have an isomorphism of graded groups:
\[ \tilde{\Kth}_{K(A)}^*(X(A)) = \mbox{R}_T^{\varnothing}[\beta^{\pm1}],  \]
where $\beta$ is the Bott class in degree 2, and $\mbox{R}_T^{\varnothing}$ is graded so as to belong entirely in degree $n$. 
\end{thm}
\begin{remark}
Note that we may identify $\mbox{R}_T^{\varnothing}$ with the Grothendieck group of irreducible highest weight representations of $K(A)$, as described earlier.
Under this identification, we will goemetrically interpret the above theorem using an equivariant family of  cubic Dirac operators.
\end{remark}
 
\noindent
In Section 6, we show that:

\begin{thm} Let $K(A)$ denote a Kac-Moody group of extended compact type, where $|I_0| = n+1$ for some integer $n>1$. Let $A_0 = (a_{i,j})_{i,j \in I_0}$ denote the sub Cartan matrix of compact type. Then the action of $K(A_0)$ on $X(A_0)$ extends to $K_{I_0}(A)$, and the following restriction map is an injection:
 \[  \tilde{\Kth}^*_{K(A)}(X(A)) \longrightarrow {}^A \tilde{\Kth}^*_{K_{I_0}(A)}(X(A_0)),  \]
 where reduced Dominant K-theory denotes the kernel along any $T$-orbit. Moreover, the maximal dominant $K(A)$-representation restricts to a dominant $K_{I_0}(A)$-representation, inducing:
 \[ \mbox{St}: \,  {}^A\tilde{\Kth}^*_{K_{I_0}(A)}(X(A_0)) \longrightarrow \tilde{\Kth}^*_{K_{I_0}(A)}(X(A_0)).  \]  The map $\mbox{St}$ is injective, and has image given by the ideal generated by characters that correspond to weights in the interior of the Weyl chamber of $K_{I_0}(A)$, which are in the $W(A)$-orbit of dominant characters of $K(A)$. 
 
 \noindent
 The image of $\tilde{\Kth}^*_{K(A)}(X(A))$ inside $\tilde{\Kth}^*_{K_{I_0}(A)}(X(A_0))$ may be identified with the characters described above which correspond to weights that also belong to the antidominant Weyl chamber of $K_{J_0}(A)$. 
 \end{thm}
\section{Structure of the space $X(A)$}
\noindent
Before we begin a detailed study of the space $X(A)$ that will allow us prove that $X(A)$ is equivalent to the classifying space of proper $K(A)$- actions, let us review the structure of the Coxeter group $W(A)$. Details can be found in \cite{H}. The group $W(A)$ is defined as follows:
\[ W(A) = \langle r_i, \, i \in I \, | \, (r_i r_j)^{m_{i,j}} = 1 \rangle,  \]
where the integers $m_{i,j}$ depend on the product of the entries $a_{i,j}a_{j,i}$ in the generalized Cartan matrix $A = (a_{i,j})$. In particular, $m_{i,i} = 1$. The word length with respect to the generators $r_i$ defines a length function $l(w)$ on $W(A)$. Moreover, we may define a partial order on $W(A)$ known as the Bruhat order. Under the Bruhat order, we say $v \leq w$ if $v$ may be obtained from some (in fact any!) reduced expression for $w$ by deleting some generators. 

\noindent
Given any two subsets $J,K \subseteq I$, let $W_J(A)$ and $W_K(A)$ denote the subgroups generated by the elements $r_j$ for $j \in J$ (resp. $K$). Let $w \in W(A)$ be an arbitrary element. Consider the double coset $W_J(A)\, w \, W_K(A) \subseteq W(A)$. This double coset contains a unique element $w_0$ of minimal length. Moreover, any other element $u \in W_J(A)\, w \, W_K(A)$ can be written in reduced form: $u = \alpha \, w_0 \, \beta$, where $\alpha \in W_J(A)$ and $\beta \in W_K(A)$. It should be pointed out that the expression $\alpha \, w_0 \, \beta$ above, is not unique. We have $\alpha \, w_0 = w_0 \, \beta$ if and only if $ \alpha \in W_L(A)$, where $W_L(A) = W_J(A)\, \cap \, w_0 \, W_K(A) \, w_0^{-1}$ (any such intersection can be shown to be generated by a subset $L \subseteq J$). 

\noindent
We denote the set of minimal $W_J(A)$-$W_K(A)$ double coset representatives by $^JW^K$. If $K = \varnothing$, then we denote the set of minimal left $W_J(A)$-coset representatives by $^J W$. The following claim follows from the above discussion:
\begin{claim} \label{minrep} Let $J,K$ be any subsets of $I$. Let $w$ be any element of $^JW$. Then $w$ belongs to the set $^JW^{K}$ if and only if $l(w\, r_k) > l(w)$ for all $k \in K$. 
\end{claim}
\noindent
Let us now recall the generalized Bruhat decomposition for the split Kac-Moody group $G(A)$. Given any two subsets $J,K \subseteq I$. Let $G_J(A)$ and $G_K(A)$ denote the corresponding parabolic subgroups. Then the group $G(A)$ admits a Bruhat decomposition into double cosets:
\[ G(A) = \coprod_{w \in ^JW^K} G_J(A)\, \tilde{w} \, G_K(A),  \]
where $\tilde{w}$ denotes any element of $N(T)$ lifting $w \in W(A)$. The closure of a double coset $G_J(A) \, \tilde{w} \, G_K(A)$ is given by:
\[ \overline{ G_J(A) \, \tilde{w} \, G_K(A)} = \coprod_{v \in ^JW^K, \, v\leq w} G_J(A)\, \tilde{v} \, G_K(A).  \]
One may decompose the subspace $G_J(A) \, \tilde{w} \, G_K(A)$ further as
\[ G_J(A) \, \tilde{w} \, G_K(A) = \coprod_{v \in W_J(A) \, w \, W_K(A)} B\, \tilde{v} \, B,  \]
where $B = G_{\varnothing}(A)$ denotes the positive Borel subgroup. The subspace $B\, \tilde{v} \, B$ has the structure of the right $B$ space $\C^{l(v)} \times B$ and one has an isomorphism:
\[ B \, \tilde{v} \, B = \prod_{i = 1}^s B\, \tilde{r}_s \, B,  \]
where $v = r_1\ldots r_s$ is a reduced decomposition. The above structure in fact gives the homogeneous space $G(A)/G_K(A)$ the structure of a CW complex, and as such it is homeomorphic to the corresponding homogeneous space for the unitary forms: $K(A)/K_K(A)$

\medskip
\noindent
Now recall the space $X(A)$ defined in the previous section:
\[ X(A) = \frac{K(A)/T \times |\gC(A)|}{\sim},   \]
where we identify $(gT,x)$ with $(hT,y)$ iff $x=y \in \Delta_J(n)$, and $g=h \, \mbox{mod} \, K_J(A)$. The space $|\gC(A)|$ was the subcomplex of the barycentric subdivision of the $n$-simplex consisting of those faces indexed on chains $\varnothing \subseteq J_0 \subset J_1 \ldots \subset J_m \subseteq I$ for which $J_k \in \gC(A)$ for all $k$. 

\noindent
By a theorem of Kac-Peterson \cite{K3}, any compact subgroup of $K(A)$ is conjugate to a subgroup of $K_J(A)$ for some $J \in \gC(A)$. Hence, to show equivariant contractibility of $X(A)$ with respect to any compact subgroup of $K(A)$, it is sufficient to show that $X(A)$ is $K_J(A)$-equivariantly contractible for any $J \in \gC(A)$. And so we require:

\begin{thm} \label{Tits}
Let $J \in \gC(A)$. Then the space $X(A)$ is $K_J(A)$-equivariantly contracible. 
\end{thm}
\begin{proof}
 Define a $K_J(A)$ equivariant filtration of $G(A)/B$ by finite complexes:
\[ (G(A)/B)_k= \coprod_{w \in ^JW, \, l(w) \leq k} G_J(A)\, \tilde{w} \, B/B , \quad k \geq 0. \]
Identifying $K(A)/T$ with $G(A)/B$, we get a filtration of $X(A)$ by $K_J(A)$ invariant finite dimensional sub complexes:
\[ X_k(A) = \frac{(G(A)/B)_k \times |\gC(A)|}{\sim}, \, k>0, \quad  X_0(A) = \frac{G_J(A)/B \times |\gC_J(A)|}{\sim},  \]
where $|\gC_J(A)|$ denotes the subspace $|\gC(A)| \cap \Delta_J(n)$, which is the realization of the nerve of the subcategory under the object $J \in \gC(A)$ (and is hence contractible). Notice that $X_0(A)$ is in fact a trivial $G_J(A)$-space homeomorphic to $|\gC_J(A)|$. We now proceed to show that $X_k(A)$ equivariantly deformation retracts onto $X_{k-1}(A)$ thereby showing equivariant contractibility.

\medskip
\noindent
An element $(gB,x) \in X_k(A)$ belongs to $X_{k-1}(A)$ if either $gT \in (G(A)/B)_{k-1}$ or if $x \in \Delta_K(n)$ and $gG_K(A) = hG_K(A)$ for some $h$ such that $hB \in (G(A)/B)_{k-1}$. This observation may be expressed as:
\[ X_k(A)/X_{k-1}(A) = \bigvee_{w \in ^JW, \,l(w)=k} (G_J(A) \, \tilde{w} \, B/B)_+ \wedge (|\gC(A)|/|\gC_w(A)|),  \]
where $(G_J(A) \, \tilde{w} \, B/B)_+$ indicates the one point compactification of the space, and $|\gC_w(A)|$ is a subcomplex of $|\gC(A)|$ consisting of faces indexed on chains $\varnothing \subseteq J_0 \subset J_1 \ldots \subset J_m \subseteq I$ for which $w$ does not belong to $^JW^{J_0}$. Hence, to show that $X_k(A)$ equivariantly retracts onto $X_{k-1}(A)$ it is sufficient to show that there is a deformation retraction from $|\gC(A)|$ to $|\gC_w(A)|$ for all $w \in {^JW}$. It is enough to show that the subspace $|\gC_w(A)|$ is contractible. 

\medskip
\noindent
Let $I_w \subseteq I$ be the subset defined as $I_w = \{ i \in I \, | \, l(wr_i) < l(w) \} $. 
By studying the coset $w W_{I_w}(A)$, it is easy to see that $I_w \in \gC(A)$ (see \cite{D3}(4.7.2)). Using claim \ref{minrep} we see that $|\gC_w(A)|$ is the geometric realization of the nerve of the subcategory:
\[ \gC_w(A) = \{ J \in \gC(A) \, | \, J \cap I_w \neq \varnothing \}. \]
Note that $\gC_w(A)$ is clearly equivalent to the subcategory consisting of subsets that are contained in $I_w$, which has a terminal object (namely $I_w$). Hence the nerve of $\gC_w(A)$ is contractible. 
\end{proof}

\begin{remark} \label{DC}
One may consider the $T$-fixed point subspace of $X(A)$, where $T$ is the maximal torus of $K(A)$. It follows that the $X(A)^T$ is a proper $W(A)$-space. Moreover, it is the classifying space for proper $W(A)$-actions. This space has been studied in great detail by M. W. Davis and coauthors \cite{D1,D2,D3}, and is sometimes called the Davis complex $\Sigma$, for the Coxeter group $W(A)$. It is not hard to see that 
\[ \Sigma =  X(A)^T = \frac{W(A) \times |\gC(A)|}{\sim},   \]
where we identify $(w,x)$ with $(v,y)$ iff $x=y \in \Delta_J(n)$, and $w=v \, \mbox{mod} \, W_J(A)$. As before, this may be expressed as a homotopy colimit of a suitable functor defined on the category $\gC(A)$ taking values in the category of $W(A)$-spaces. 
\end{remark}
\section{Dominant K-theory}
\noindent
In this section we study Dominant K-theory in detail. For the sake of simplicity, we assume throughout this section that the generalized Cartan matrix $A$ is symmetrizable. Under this assumption, any irreducible highest  weight representation of $K(A)$ is unitary. We believe that this assumption is not strictly necessary, and that most of our constructions may be extended to arbitrary generalized Cartan matrices. 

\noindent
Now recall some definitions introducted in Section 2. We say that a unitary representation of $K(A)$ is dominant if it decomposes as a completed sum of irreducible highest weight representations. Hence one has a maximal dominant representation $\mathcal{H}$ obtained by taking a completed sum of (countably many copies of) all highest weight representations. Since $K(A)$ is the amalgamated product (in the category of topological groups) of compact Lie groups, we obtain a continuous map from $K(A)$ to the group $\mathbb{U}(\mathcal{H})$ of unitary operators on $\mathcal{H}$, with the compact open topology.

\noindent
Now given any highest weight representation $\mbox{L}_{\mu}$, one may consider the set of weights in $\gH_{\Z}^*$ for $\mbox{L}_{\mu}$. It is known \cite{Ku} that this set is contained in the convex hull of the orbit: $W(A)\, \mu$. Hence we see that the set of weights of any proper highest weight representation $\mbox{L}_{\mu}$ belongs to the set of weights contained in the Tits cone Y. Conversely, any weight in Y is the $W(A)$ translate of some weight in the Weyl chamber C, and so we have:

\begin{claim} The set of distinct weights that belong to the maximal dominant representation $\mathcal{H}$ is exactly the set of weights in the Tits cone $\mbox{Y}$. Moreover, this set is closed under addition. 
\end{claim}
\begin{defn}
For $G \subseteq K(A)$, a compact Lie subgroup, let the dominant representation ring of $G$ denoted by $\mbox{DR}_G$, be the free abelian group on the set of isomorphism classes of irreducible $G$ representations belonging to $\mathcal{H}$. This group admits the structure of a subring of the representation ring $\mbox{R}_G$ of $G$. 
\end{defn}
\noindent
Recall from Section 2 that $\Kk$ denotes the periodic K-theory spectrum represented by a model $\mathcal{F}(\mathcal{H})$ for the space of Fredholm operators on $\mathcal{H}$. 

\noindent
In the next few results, we explore the structure of Dominant K-theory. We begin with a comparison between Dominant K-theory and standard equivariant K-theory with respect to compact subgroups. 

\noindent
Let $\mbox{K}^*_G(X)$ denote standard equivariant K-theory modeled on a the space of Fredholm operators on a $G$-stable Hilbert space $\mathcal{L}(G)$ \cite{S}. Since, by definition $\mathcal{L}(G)$ contains every $G$-representation infinitely often, we may fix an equivariant isometry $\mathcal{H} \subseteq \mathcal{L}(G)$ (notice that any two isometries are equivariantly isotopic). Then we have:

\begin{claim}
Let $G \subseteq K(A)$ be a compact subgroup. Then the stabilization map is an injection:
\[ \mbox{St}: {}^A \Kth^*_G(S^0) \longrightarrow \mbox{K}^*_G(S^0). \]
Furthermore, given a proper orbit $X = K(A)_+\wedge_G S^0$ for some compact Lie subgroup $G \subseteq K(A)$, we have a canonical isomorphism:
\[ \Kth^*_{K(A)}(X) = \mbox{DR}_G[\beta^{\pm1}],   \]
where $\beta$ is the Bott class in degree 2, and $\mbox{DR}_G$ lies in degree 0. In particular, the odd Dominant K-cohomology of $X$ is trivial. 

\end{claim}
\begin{proof}
Let $\mathcal{M}$ be a minimal $G$-invariant complement of $\mathcal{H}$ inside $\mathcal{L}(G)$, so that $\mathcal{H}\oplus \mathcal{M}$ is an $G$-stable Hilbert space inside $\mathcal{L}(G)$. In particular, $\mathcal{M}$ and $\mathcal{H}$ share no nonzero $G$-representations. On the level of $G$-fixed points on the space of Fredholm operators, we have an inclusion: 
\[ \mathcal{F}^G(\mathcal{H} \oplus \mathcal{M}) = \mathcal{F}^G(\mathcal{H}) \times \mathcal{F}^G(\mathcal{M}) \longrightarrow \mathcal{F}^G (\mathcal{L}(G)). \]
Moreover, it is easy to see that this map is a homotopy equivalence. Hence, we have an injection:
\[ \mbox{St}: {}^A \Kth^0_G(S^m) = \pi_m\mathcal{F}^G(\mathcal{H}) \longrightarrow \pi_m\mathcal{F}^G(\mathcal{L}(G)) = \mbox{K}^0_G(S^m).  \]
Furthermore, decomposing $\mathcal{H}$ into its irreducible isotypical summands, shows that the group $\pi_0 (\mathcal{F}^G(\mathcal{H}))$ is free on the irreducible $G$-summands in $\mathcal{H}$. In other words, it is isomorphic to $\mbox{DR}_G$. This gives the identification of $\Kth^0_{K(A)}(X)$ we claimed. 
\end{proof} 

\noindent
The following theorem may be seen as a Thom isomorphism theorem for Dominant K-theory. It would be interesting to know the most general conditions on an equivariant vector bundle that ensure the existence of a Thom class. 
\begin{thm} \label{Thom}
Let $G \subseteq K(A)$ be a subgroup of the form $K_J(A)$ for some $J \in \gC(A)$, and let $r$ be the rank of $K(A)$. Let $\gG$ denote the Adjoint representation of $K_J(A)$. Then there exists a fundamental irreducible representation $\lambda$ of the Clifford algebra Cliff ($\gG \otimes \C$) that serves as a Thom class in ${}^A \Kth^r_G(S^{\gG})$ (a generator for ${}^A \Kth^*_G(S^{\gG})$ as a free module of rank one over $\mbox{DR}_G$), where $S^{\gG}$ denotes the one point compactification of $\gG$. 
\end{thm}
\begin{proof} 
We begin by giving an explicit description of $\lambda$. Fix an invariant inner product B on $\gG$, and let Cliff($\gG \otimes \C$) denote the corresponding complex Clifford algebra. One has the triangular decomposition: $\gG \otimes \C = \eta_+ \oplus \eta_- \oplus \gH$, where $\eta_{\pm}$ denote the nilpotent subalgebras. The inner product extends to a Hermitian inner product on $\gG \otimes \C$, which we also denote B, and for which the triangular decomposition is orthogonal . 

\noindent
Recall that $\gH$ contains the lattice $\gH_{\Z}$ containing the coroots $h_i$ for $i \in I$. Fix a dual set $h_i^* \in \gH_{\Z}^*$. We may decompose $\eta_{\pm}$ further into root spaces indexed on the roots generated by the simple roots in the set $J$:
\[ \eta_{\pm} = \sum_{\alpha \in \Delta_{\pm}} \gG_{\alpha},  \]
where $\Delta_{\pm}$ denotes the positive (resp. negative) roots for $K_J(A)$. Now fix a Weyl element $\rho_J$, defined by:
\[ \rho_J = \sum_{j \in J} h_j^*. \]
It is easy to see using character theory that the irreducible $G$-module $\mbox{L}_{\rho_J}$, with highest weight $\rho_J$ belongs to $\mathcal{H}$, and has character:
\[ \mbox{ch} \, \mbox{L}_{\rho_J}  = e^{\rho_J} \prod_{\alpha \in \Delta_+} (1+ e^{-\alpha}) = e^{\rho_J} \mbox{ch} \, \Lambda^*(\eta_-),  \]
where $\Lambda^*(\eta_-)$ denotes the exterior algebra on $\eta_-$. In particular, the vector space $\mbox{L}_{\rho_J}$ is naturally $\Z/2$-graded and belongs to $\mathcal{H}$. The exterior algebra $\Lambda^*(\eta_-)$ can naturally be identified with the fundamental Clifford module for the Clifford algebra Cliff($\eta_+ \oplus \eta_-$). Let ${\mathbb S}(\gH)$ denote an irreducible Clifford module for Cliff($\gH$). It is easy to see that the action of Cliff($\gG \otimes \C$) on ${\mathbb S}(\gH) \otimes \Lambda^*(\eta_-)$ extends uniquely to an action of $G \ltimes \mbox{Cliff}(\gG \otimes \C)$ on $\lambda$, with highest weight $\rho_J$, where:
\[ \lambda = \C_{\rho_J} \otimes {\mathbb S}(\gH) \otimes \Lambda^*(\eta_-) = {\mathbb S}(\gH) \otimes \mbox{L}_{\rho_J}. \]
The Clifford multiplication parametrized by the base space $\gG$, naturally describes $\lambda$ as an element in ${}^A \Kth^r_G(S^{\gG})$. 

\noindent
We now show that $\lambda$ is a free generator of rank one for ${}^A \Kth^r_G(S^{\gG})$, as a $\mbox{DR}_G$-module. Let $\mbox{L}_\mu$ be an irreducible generator of $\mbox{DR}_G$. Consider the element $\lambda \otimes \mbox{L}_{\mu} \in {}^A\Kth^r_G(S^{\gG})$, and restrict it to ${}^A\Kth^r_T(S^{\gH_{\R}})$, along the action map:
\[ \varphi: G_+ \wedge_T S^{\gH_{\R}} \longrightarrow S^{\gG}, \]
where $T \subseteq G$ is the maximal torus. Identifying ${}^A\Kth^r_T(S^{\gH_{\R}})$, with $\mbox{DR}_T$, and using the character formula, we see that $\lambda \otimes \mbox{L}_{\mu}$ has virtual character given by:
\[ \sum_{w \in W_J(A)} (-1)^w e^{w(\rho_J + \mu)}. \]
This correspondence shows that ${}^A \Kth_G^r(S^{\gG})$ contains a subgroup $\mbox{DR}^+_G$ generated by positive dominant characters $\tau$, with the property: $\{  \tau \in \mbox{D} \, | \, \tau(h_j) > 0, \,  j \in J \}$. 

\medskip
\noindent
 It remains to show that all elements in ${}^A \Kth_G^r(S^{\gG})$ are in this subgroup. For this we will make an explicit computation of ${}^A \Kth_G^r(S^{\gG})$ using a homotopy decomposition of $S^{\gG}$ as a $G$-space as constructed in \cite{CK}. This method of computation uses the Bousfield-Kan spectral sequence for the cohomology of a homotopy colimit. This spectral sequence will be used extensively throughout this paper. 
 
 \medskip
 \noindent
 It will be more convenient to study the unit sphere $S(\gG)$ in the representation $\gG$. Notice that one has an equivariant cofiber sequence: 
 \[ S(\gG)_+ \longrightarrow S^0 \longrightarrow S^{\gG}. \]
 Therefore, the calculation of ${}^A\Kth_G(S^{\gG})$ will follow from a similar calculation for $S(\gG)$.

 \medskip
 \noindent
 It is shown in \cite{CK} that $S(\gG)$ is a suspension of the following space (the suspension coordinates correspond to the rank of the center of $G$):
 \[ X(G) := \hocolim_{S \subsetneq J} G/K_I(A). \]
 Filtering $X(G)$ by the equivariant skeleta, we get a convergent spectral sequence $(E_n, d_n)$, $|d_n| = (n,1-n)$, and with $E_2$ term given by:
\[ E_2^{p,*} = \varprojlim  {}^p {}^A\Kth_{G}^*(G/K_{\bullet}(A)) = \varprojlim  {}^p \mbox{DR}_{\bullet}[\beta^{\pm1}] \Rightarrow {}^A\Kth_{G}^p(X(G)))[\beta^{\pm1}], \]
where $\beta$ denotes the invertible Bott class, and we have simplified the notation $\mbox{DR}_{\bullet}$ to denote the functor $S \mapsto \mbox{DR}_{K_S(A)}$ for $S \subsetneq J$.

\medskip
\noindent
Now using character formula, we see that $\mbox{DR}_S$ is isomorphic to the $W_S(A)$-invariant characters:  $\mbox{DR}_T^{W_S(A)}$. In other words, we have: 
\[ \mbox{DR}_S = \mbox{DR}_T^{W_S(A)} = \Hom_{W_J(A)}(W_J(A)/W_S(A), \mbox{DR}_T). \]
Consider the $W_J(A)$-equivariant spherical Davis complex: $\Sigma_J$ defined as: 
\[ \Sigma_J = \hocolim_{S \subsetneq J} W_J(A)/W_S(A). \]
 It follows that $\varprojlim  {}^p \mbox{DR}_{\bullet}$ is canonically isomorphic to the equivariant cohomology (as defined in \cite{D2}) of $\Sigma_J \subset \gH$, with values in the ring $\mbox{DR}_T$: 
 \[ \varprojlim {}^p \mbox{DR}_{\bullet} = \mbox{H}^p_{W_J(A)}(\Sigma_J, \mbox{DR}_T). \] 

\medskip
\noindent
Now recall that the set of dominant weights $\mbox{D}$ has a decomposition indexed by subsets $K \subseteq I$:
\[ \mbox{D} = \coprod  \mbox{D}_K, \quad \mbox{where} \quad \mbox{D}_K = \{ \lambda \in \mbox{D} \, | \, \lambda(h_k)=0, \iff k \in K \}. \]
Let $\mbox{R}_T^K$ denote the ideal in $\mbox{DR}_T$ generated by the weights belonging to the subset $\mbox{D}_K$. We get a corresponding decomposition of $\mbox{DR}_T$ as a $W(A)$-module:
\[ \mbox{DR}_T = \bigoplus \mbox{DR}_T^K, \quad \mbox{where} \quad \mbox{DR}_T^K \cong \Z[W(A)/W_K(A)] \otimes \mbox{R}_T^K. \]
We therefore have an induced decomposition of the functor $\mbox{DR}_{\bullet} = \bigoplus \mbox{DR}_{\bullet}^K$ indexed by $K \subseteq I$. On taking derived functors we have:
\[ \varprojlim {}^p \mbox{DR}_{\bullet} = \bigoplus \varprojlim {}^p \mbox{DR}_{\bullet}^K = \bigoplus \mbox{H}^p_{W_J(A)}(\Sigma_J, \Z[W(A)/W_K(A)]) \otimes \mbox{R}_T^K. \] 
Now the left $W_J(A)$-space $W(A)/W_K(A)$ is a disjoint union over double cosets: 
\[ W(A)/W_K(A) = \coprod_{w \in {}^{J} W {}^K} W_{J}(A)/W_{K_w}(A), \quad \mbox{where} \quad W_{K_w}(A) = W_{J}(A) \cap \, w \, W_K(A) \, w^{-1}.  \]
Since $\Sigma_J$ is a compact simplicial complex, it is easy to see directly or using \cite{D1, D2} that:
\[ \mbox{H}^p_{W_J(A)}(\Sigma_J, \Z[W_J(A)]) = \mbox{H}^p(\Sigma_J, \Z) = 0, \, \,  \mbox{if} \, \, p\neq \{0, |J|-1\}, \, \,  \mbox{and} \, \,   = \, \Z \, \, \mbox{if} \, \, p=\{0, |J|-1\}. \]
In addition, given a non-empty subset  $K_w \subseteq J$, 
\[  \mbox{H}^p_{W_J(A)}(\Sigma_J, \Z[W_J(A)/W_{K_w}(A)]) = \mbox{H}^p(\Sigma_J/W_{K_w}(A), \Z) = 0 \, \, \mbox{if} \, \, p \neq 0, \, \, = \Z \, \, \mbox{if} \, \, p=0. \] 
The last equality above follows form the fact that $\Sigma_J/W_{K_w(A)}$ is the fundamental domain of the $W_{K_w(A)}$-action on $\Sigma_J$. Hence it is the intersection of $\Sigma_J$ with a cone that lies in a half-quadrant. In particular, it is a retract of the cone and therefore contractible. 

\medskip
\noindent
It follows that the spectral sequence has only two columns $p=\{0, |J|-1\}$, with:
\[ E_2^{0,\ast} = \bigoplus_{\{K, {}^{J} W {}^K\}} \mbox{R}_T^K[\beta^{\pm1}] = \mbox{DR}_G[\beta^{\pm1}], \quad \quad E_2^{|J|-1,\ast} = \bigoplus_{\{K, {}^{J} \overline{W} {}^K\}} \mbox{R}_T^K[\beta^{\pm1}] = \mbox{DR}^+_G[\beta^{\pm1}],\]
 where ${}^{J} \overline{W} {}^K$ is the subset of elements $w \in {}^{J} W {}^K$ for which $K_w = \emptyset$. This spectral sequence must collapse because $\mbox{DR}_G$ is detected by the pinch map $S(\gG)_+ \longrightarrow S^0$. It follows that ${}^A\Kth^\ast_G(S^{\gG})$ is isomorphic to $\mbox{DR}_G^+[\beta^{\pm1}]$, which is what we wanted to show. 
 
 \end{proof}

 \begin{remark} \label{PF}
 The above Thom isomorphism allows us to define the pushforward map in Dominant K-theory for inclusions between the groups $K_J(A)$ induced by morphisms in $\gC(A)$. For example, let $\iota : T \rightarrow K_J(A)$ be the canonical inclusion of the maximal torus. Then we have the (surjective) pushforward map given by Dirac Induction: 
 \[ \iota_J : \mbox{DR}_T \longrightarrow \mbox{DR}_J, \]
 where $\mbox{DR}_J$ denotes the dominant representation ring of $K_J(A)$. Moreover, we also recover the character formula as the composite:
 \[ \iota^* \iota_J (e^{\mu}) = \frac{\sum_{w \in W_J(A)} (-1)^w e^{w(\mu)}} {e^{\rho_J} \prod_{\alpha \in \Delta_+}(1-e^{-\alpha})}.\]
 \end{remark}

 \section{The Dominant K-theory for the compact type}
\noindent
Fix a Kac-Moody group $K(A)$ of compact type, with maximal torus $T$. In this section we describe the Dominant K-cohomology groups of the space $X(A)$. 

\medskip
\noindent
Recall that $X(A)$ can be written as a proper $K(A)$-CW complex given by:
\[ X(A) = \frac{K(A)/T \times \Delta(n)}{\sim},  \]
where we identify $(gT,x)$ with $(hT,y)$ iff $x=y \in \Delta_J(n)$, and $g=h \, \mbox{mod} \, K_J(A)$. 

\noindent
Define the pinch map induced by the projection: $\Delta(n) \rightarrow \Delta(n)/\partial \Delta(n) = S^n$:
\[ \pi: X(A) \longrightarrow K(A)_+\wedge_T S^n.  \]
Assume that the generalized Cartan matrix $A$ has size $n+1$, and that $K(A)$ is not of finite type. Let $\mbox{R}_T^{\varnothing}$ denote the ideal within the character ring of $T$ generated by the regular dominant weights, $\mbox{D}_+$. In particular, we may identify $\mbox{R}_T^{\varnothing}$ with the Grothendieck group of irreducible highest weight representations of $K(A)$. 

\medskip
\noindent
For the space $X(A)$, define the reduced Dominant K-theory $\tilde{\Kth}_{K(A)}^*(X(A))$, to be the kernel of the restriction map along any $T$-the orbit in $X(A)$. 
 \begin{thm} \label{thm1}
We have an isomorphism of graded groups:
\[ \tilde{\Kth}_{K(A)}^*(X(A)) = \mbox{R}_T^{\varnothing}[\beta^{\pm1}],  \]
where $\beta$ is the Bott class in degree 2, and $\mbox{R}_T^{\varnothing}$ is graded so as to belong entirely in degree $n$. Moreover, the identification of $\tilde{\Kth}_{K(A)}^n(X(A))$ with $\mbox{R}_T^{\varnothing}$ is induced by the pinch map $\pi$.
\end{thm}
\begin{proof}
Since $K(A)$ is of compact type (but not of finite type), the category $\gC(A)$ is the category of all proper subsets of $I$. We may see $X(A)$ as a homotopy colimit of a diagram over $\gC(A)$ taking the value $K(A)/K_J(A)$ for $J \in \gC(A)$.  

\noindent
Filtering $X(A)$ by the equivariant skeleta, we get a convergent spectral sequence $(E_n, d_n)$, $|d_n| = (n,n-1)$, and with $E_2$ term given by:
\[ E_2^{p,*} = \varprojlim  {}^p \Kth_{K(A)}^*(K(A)/K_{\bullet}(A)) = \varprojlim  {}^p \mbox{DR}_{\bullet}[\beta^{\pm1}] \Rightarrow \Kth_{K(A)}^p(X(A))[\beta^{\pm1}], \]
where $\mbox{DR}_{\bullet}$ denotes the functor $J \mapsto \mbox{DR}_{K_J(A)}$, and $\beta$ denotes the invertible Bott class of degree two. Now using remark \ref{PF}, it is easy to see that $\mbox{DR}_J$ is isomorphic to the $W_J(A)$-invariant characters:  $\mbox{DR}_T^{W_J(A)}$. Notice also that 
\[ \mbox{DR}_T^{W_J(A)} = \Hom_{W(A)}(W(A)/W_J(A), \mbox{DR}_T). \]
 It follows that $\varprojlim  {}^p \mbox{DR}_{\bullet}$ is canonically isomorphic to the equivariant cohomology (as defined in \cite{D2}) of the Davis complex $\Sigma$ (\ref{DC}), with values in the ring $\mbox{DR}_T$: 
 \[ \varprojlim {}^p \mbox{DR}_{\bullet} = \mbox{H}^p_{W(A)}(\Sigma, \mbox{DR}_T). \] 
Now recall that the set of dominant weights $\mbox{D}$ has a decomposition indexed by subsets $K \subseteq I$:
\[ \mbox{D} = \coprod  \mbox{D}_K, \quad \mbox{where} \quad \mbox{D}_K = \{ \lambda \in \mbox{D} \, | \, \lambda(h_k)=0, \iff k \in K \}. \]
Let $\mbox{R}_T^K$ denote the ideal in $\mbox{DR}_T$ generated by the weights belonging to the subset $\mbox{D}_K$. We get a corresponding decomposition of $\mbox{DR}_T$ as a $W(A)$-module:
\[ \mbox{DR}_T = \bigoplus \mbox{DR}_T^K, \quad \mbox{where} \quad \mbox{DR}_T^K \cong \Z[W(A)/W_K(A)] \otimes \mbox{R}_T^K. \]
We therefore have an induced decomposition of the functor $\mbox{DR}_{\bullet} = \bigoplus \mbox{DR}_{\bullet}^K$ indexed by $K \subseteq I$. On taking derived functors we have:
\[ \varprojlim {}^p \mbox{DR}_{\bullet} = \bigoplus \varprojlim {}^p \mbox{DR}_{\bullet}^K = \bigoplus \mbox{H}^p_{W(A)}(\Sigma, \Z[W(A)/W_K(A)]) \otimes \mbox{R}_T^K. \] 
We now proceed by considering three cases: $K = I, \varnothing$, and the remaining cases. 

\noindent
{\bf The case K=I}. Observe that $W_I(A) = W(A)$, and that $\mbox{R}_T^I$ is the subring of $\mbox{DR}_T$ generated by the $W(A)$-invariant weights $\mbox{D}_I$. Moreover, we have: 
\[ \mbox{H}^p_{W(A)}(\Sigma, \Z) = \mbox{H}^p(\Delta(n), \Z) = 0, \, \,  \mbox{if} \, \, p>0, \, \,  \mbox{and} \, \,   \Z \, \, \mbox{if} \, \, p=0. \]
Hence $\varprojlim {}^0 \mbox{DR}_{\bullet}^I = \mbox{R}_T^I$, and all higher inverse limits vanish. Now every weight in $\mbox{D}_I$ represents an element of $\Kth_{K(A)}^0(X(A))$ given by a one dimensional highest weight representation of $K(A)$, hence they represent permanent cycles. Moreover, these weights are detected under any $T$-orbit in $X(A)$. 

\medskip
\noindent
{\bf The case K=$\varnothing$}. Notice that $W_{\varnothing}= \{ 1 \}$, and that $\mbox{R}_T^{\varnothing}$ is the ideal in $\mbox{DR}_T$ generated by the weights in $\mbox{D}_+$. We have by \cite{D1,D2}:
\[  \mbox{H}^p_{W(A)}(\Sigma, \Z[W(A)]) = \mbox{H}^p_c(\Sigma, \Z) = 0, \, \,  \mbox{if} \, \, p\neq n, \, \,  \mbox{and} \, \,   \Z \, \, \mbox{if} \, \, p=n, \]
where $\mbox{H}^p_c(\Sigma, \Z)$ denotes the cohomology of $\Sigma$ with compact supports. 
Hence we deduce that $\varprojlim {}^n \mbox{DR}_{\bullet}^{\varnothing} = \mbox{R}_T^{\varnothing}$, and all other derived inverse limits vanish. 

\medskip
\noindent
{\bf The cases K $\neq \{I, \varnothing $\}}. If $K \subseteq I$ is a nontrivial proper subset of $I$, then we know that $W_K(A)$ is a finite subgroup of $W(A)$. The argument in \cite{D2} also shows that:
\[ \mbox{H}^p_{W(A)}(\Sigma, \Z[W(A)/W_K(A)]) = \mbox{H}^p_c(\Sigma/W_K(A), \Z). \] 
By \cite{D1} (Section 3), $\Sigma/W_K(A)$ is a proper retract of $\Sigma$. Hence the group: $\mbox{H}^p_c(\Sigma/W_K(A), \Z)$ is a summand within $\mbox{H}^p_c(\Sigma, \Z)$. From the previous case, recall that the latter group is trivial if $p \neq n$, and is free cyclic for $p=n$. Moreover, given any $ i \in I$, the reflection $r_i \in W(A)$ acts by reversing the sign of the generator of $\mbox{H}^n_c(\Sigma, \Z)$. By picking $i \in K$, we see that the groups $\mbox{H}^p_c(\Sigma/W_K(A), \Z)$ must be trivial for all $p$. 

\medskip
\noindent
From the above three cases we notice that the spectral sequence must collapse at $E_2$, showing that the groups $\tilde{\Kth}_{K(A)}^{*+n}(X(A))$ are isomorphic to $\mbox{R}_T^{\varnothing}[\beta^{\pm 1}]$. Since the case $J = \varnothing$ corresponds to the lowest filtration in the $E_{\infty}$ term, it can be identified with the image of the map given by pinching off the top cell: $\pi :  X(A) \rightarrow K(A)_+ \wedge_T S^n$. 
\end{proof}
 
 \section{The Geometric identification}
 \noindent
 For a compact type Kac-Moody group, we would like to give a geometric meaning to the image of $\pi$ in terms of Fredholm families. One may see this as a global Thom isomorphism theorem for Dominant K-theory (in the spirit of the local Thom isomorphism theorem developed in section 4). That such a result could hold globally, was first shown in the Affine case in \cite{FHT} using the cubic Dirac operator. We will adapt their argument to our context. 
  
\medskip
\noindent
 Let $A$ be a symmetrizable generalized Cartan matrix, and let $\gG(A)$ be the corresponding complex Lie algebra with Cartan subalgebra $\gH$. Recall the triangular decomposition:
  \[ \gG(A) = \eta_+ \oplus \eta_- \oplus \gH. \]
Let $\gG(A)^*$ denote the restricted dual of $\gG(A)$. So $\gG(A)^*$ is a direct sum of the duals of the individual (finite dimensional) root spaces. 

\noindent
Let $K_J(A)$ be a compact subgroup of $K(A)$ for some $J \in \gC(A)$, and let $\mbox{B}$ be a $K_J(A)$-invariant Hermitian inner product on $\gG(A)$. Let $\mbox{S} = \mbox{B}^{\omega}$ denote the non degenerate, symmetric, bilinear form on $\gG(A)$ given by the $\omega$-conjugate of $\mbox{B}$, defined as $\mbox{S}(x,y) = \mbox{B}(x,\omega(y))$. We identify $\gG(A)^*$ with $\gG(A)$ via S, and endow it with the induced form. The subspaces $\eta_{\pm}^*$ are S-isotropic, and can be seen as a polarization of $\gG(A)^*$ \cite{PS}. Let Clif($\gG(A)^*$) denote the complex Clifford algebra generated by $\gG(A)^*$, modulo the relation $x^2 = \frac{1}{2} S(x,x)$ for $x \in \gG(A)^*$. Let ${\mathbb S}(A)$ denote the fundamental irreducible representation of Cliff($\gG(A)^*$) endowed with the canonical Hermitian inner product, which we also denote by B. We have a decomposition of algebras: \[ \mbox{Cliff}(\gG(A)^*) = \mbox{Cliff}(\gH^*) \otimes \mbox{Cliff}(\eta_+^* \oplus \eta_-^*)  \] 
Let ${\mathbb S}(\gH^*)$ denote the fundamental irreducible $\mbox{Cliff}(\gH^*)$-module. The above decomposition of $\mbox{Cliff}(\gG(A)^*)$ induces a decomposition ${\mathbb S}(A) = {\mathbb S}(\gH^*) \otimes \Lambda^*(\eta_-) $. We have identified $\Lambda^*(\eta_-)$ with the fundamental $\mbox{Cliff}(\eta_+^* \oplus \eta_-^*)$-module, with $\eta_-^*$ acting by contraction, and $\eta_+^*$ acting by exterior multiplication (once we identify $\eta_+^*$ with $\eta_-$). 

\medskip
\noindent
As in the local case, let $\mbox{L}_\rho$ denote the irreducible representation of $K(A)$ with highest weight $\rho$. From the character formula, we see that the character of $\mbox{L}_\rho$ is given by:
\[ \mbox{ch} \,  \mbox{L}_{\rho} = e^{\rho} \mbox{ch} \, \Lambda^*(\eta_{-}). \]
It is easy to see that the Cliff($\gG(A)^*$) action on $\mathbb{S}(A)$ extends uniquely to an irreducible $K_J(A) \ltimes \mbox{Cliff}(\gG(A)^*)$-module $\hat{\mathbb{S}}(A)$, with highest weight $\rho$, given by: 
\[ \hat{\mathbb{S}}(A) = \mathbb{S}(A) \otimes \C_\rho = \mathbb{S}(\gH^*) \otimes \mbox{L}_\rho. \]
Furthermore, $\hat{\mathbb{S}}(A)$ is a unitary $K_J(A)$-representation. We now proceed to construct a complex as described in \cite{FHT} (Section 14), by twisting the Koszul-Chevalley differential with the cubic Dirac operator:

\medskip
\noindent
Given a dominant weight $\mu \in \mbox{D}$, let $\mbox{L}_{\mu}$ be the corresponding irreducible unitary $K(A)$-representation with highest weight $\mu$. Let $(\mathbb{H}_\mu, \Dirac)$ denote the complex: 
\[ \mathbb{H}_\mu = \hat{\mathbb{S}}(A)\otimes \mbox{L}_\mu =  \mathbb{S}(\gH^*) \otimes \C_\rho \otimes \Lambda^*(\eta_-) \otimes \mbox{L}_{\mu},   \quad \quad \Dirac = \mbox{K} + \partial + \partial^{\dagger}, \]
where $\mbox{K}$ denotes the dirac operator on $\gH_{\R}$ with coefficients in the the $T$-representation: $\C_{\rho} \otimes \Lambda^*(\eta_-) \otimes \mbox{L}_{\mu}$. In particular, $\mbox{K}$ is given by Clifford multiplication with $\theta$ on the isotypical summand of weight $\theta$ in $\C_{\rho} \otimes \Lambda^*(\eta_-) \otimes \mbox{L}_{\mu}$. The operator $\partial$ denotes the Koszul-Chevalley differential $\partial: \Lambda^k(\eta_-) \otimes \mbox{L}_{\mu} \rightarrow \Lambda^{k-1}(\eta_-) \otimes \mbox{L}_{\mu} $, and $\partial^{\dagger}$ its adjoint with respect to the Hermitian        inner product on $\Lambda^*(\eta_-) \otimes \mbox{L}_\mu$, induced from $\mbox{L}_\mu$, and the Hermitian inner-product on $\eta_-$ induced by B. 

\medskip
\noindent
We may give an alternate description of $\Dirac$ which will be helpful later. Let $a_i$ denote a basis of root-vectors for $\eta_-$, and let $b_i$ denote the dual basis for $\eta_+$ such that $\mbox{S}(a_m,b_n) = \delta_{m,n}$. Let $h_j$ denote the coroot-basis for $\gH_{\R}$. It is easy to see that we may write the map $\partial$ acting on $\hat{\mathbb{S}}(A)\otimes \mbox{L}_\mu$ as the formal expression (using the Einstein summation notation):
\[ \partial = \psi(a_i^*) \otimes a_i +  \frac{1}{2} \, \mbox{ad}_{a_i} \, \psi(a_i^*)  \otimes \mbox{Id} \]
where $\psi(a_i^*)$ indicates Clifford action, and $\mbox{ad}_{a_i}$ denotes the adjoint action of $\eta_-$ extended to be a derivation on $\Lambda^*(\eta_-)$. It follows that we may write the elements $\partial^{\dagger}$ and $\mbox{K}$ as:
\[ \partial^{\dagger} = \psi(b_i^*)  \otimes b_i + \frac{1}{2} \,  \psi(b_i^*) \, \mbox{ad}_{a_i}^{\dagger} \otimes \mbox{Id}, \quad \quad \mbox{K} =  \psi(h_j^*) \otimes h_j+ \psi(h_j^*) \, \mbox{ad}_{h_j} \otimes \mbox{Id} + \psi(\rho) \otimes \mbox{Id} \]
With this in mind, consider the following formal expressions:
\[ \Dirac_1 = \psi(a_i^*) \otimes a_i + \psi(b_i^*) \otimes b_i + \psi(h_j^*) \otimes h_j \quad \mbox{and} \quad \quad \, \,   \]
\[  \Dirac_2 = \frac{1}{2} \,  \mbox{ad}_{a_i} \, \psi(a_i^*)   +  \frac{1}{2} \, \psi(b_i^*) \, \mbox{ad}_{a_i}^{\dagger} + \psi(h_j^*) \, \mbox{ad}_{h_j} + \psi(\rho). \]
\begin{claim} \label{dirac}
Let $\Dirac_1$ and $\Dirac_2$ be the formal expressions as given above. Then, seen as operators on $\mathbb{H}_\mu$, they are well defined and $K_J(A)$-invariant. In particular, $\Dirac = \Dirac_1 + \Dirac_2$ is $K_J(A)$-invariant. 
\end{claim}
\begin{proof}
Working with an explicit basis of elementary tensors in $\Lambda^*(\eta_-)$, it is easy to see that $\Dirac_1$, and $\Dirac_2$ are well defined.  Invariance for $\Dirac_1$ is also obvious. It remains to show that $\Dirac_2$ is $K_J(A)$-invariant, seen as an operator acting on $\hat{\mathbb{S}}(A)$.

\noindent
Define an operator $\mbox{E}$ on $\gG(A)^*$, with values in $\mbox{End}(\hat{\mathbb{S}}(A))$ as a graded commutator:
\[ \mbox{E} : \gG(A)^* \longrightarrow \mbox{End}(\hat{\mathbb{S}}(A)), \quad \mbox{E}(\tau) =  [\Dirac_2, \psi(\tau)] = \Dirac_2 \, \psi(\tau) + \Dirac_2 \, \psi(\tau), \]
where we take the graded commutator since both operators are odd. 

\noindent
Now let $\sigma$ denote the action of the Lie algebra of $K_J(A)$ on $\hat{\mathbb{S}}(A)$, and let $a$ be any element of this Lie algebra.
We proceed to show the equality (compare \cite{FHT}):
\[  [\sigma(a), \mbox{E}(\tau)] = \mbox{E}(\mbox{ad}_{a}^*(\tau)). \]
Before we prove this equality, notice that it is equivalent to $[[\Dirac_2, \sigma(a)], \psi(\tau)] = 0$ for all $\tau \in \gG(A)^*$. Hence $[\Dirac_2, \sigma(a)]$ is a scalar operator. By checking on a highest weight vector, we see that $[\Dirac_2,\sigma(a)] = 0$, or that $\Dirac_2$ is $K_J(A)$-invariant. Hence, all we need to do is to prove the above equality, which we shall prove in the next lemma.

\end{proof}

\begin{lemma}
Let $\mbox{E}$ be defined as the function:
\[ \mbox{E} : \gG(A)^* \longrightarrow \mbox{End}(\hat{\mathbb{S}}(A)), \quad \mbox{E}(\tau) =  [\Dirac_2, \psi(\tau)] = \Dirac_2 \, \psi(\tau) + \Dirac_2 \, \psi(\tau), \]
and let $a$ be an element in the Lie algebra of $K_J(A)$. Then $[\sigma(a), \mbox{E}(\tau)] = \mbox{E}(\mbox{ad}_{a}^*(\tau))$, where $\sigma$ denotes the action of $K_J(A)$ on $\hat{\mathbb{S}}(A)$.
\end{lemma}
\begin{proof}
Let us begin by expressing some structure constants using the summation convention:
\[ [ a_i, a_s] = \mbox{A}_{i,s}^k \, a_k, \quad [h_j, a_t] = -\alpha_t(h_j) \, a_t, \]
where $\alpha_t$ is the positive root corresponding to $a_t$. It follows that $\mbox{ad}_{a_i}$, $\mbox{ad}_{a_i}^{\dagger}$ and $\mbox{ad}_{h_j}$ may be written as:
\[ \mbox{ad}_{a_i} = \mbox{A}_{i,s}^k \, \psi(b_k^*) \, \psi(a_s^*), \quad \mbox{ad}_{a_i}^{\dagger} = \mbox{A}_{i,s}^k \, \psi(b_s^*) \, \psi(a_k^*), \quad \mbox{ad}_{h_j} = -\alpha_t(h_j) \, \psi(b_t^*) \, 
\psi(a_t^*). \]
We may therefore express $\Dirac_2$ in terms of the Clifford generators as:
\[ \Dirac_2 = \frac{1}{2}\mbox{A}_{i,s}^k \, \{ \psi(b_k^*) \, \psi(a_s^*) \, \psi(a_i^*) + \psi(b_i^*) \, \psi(b_s^*) \, \psi(a_k^*) \} - \alpha_t(h_j) \, \psi(h_j^*) \, \psi(b_t^*) \, \psi(a_t^*) + \psi(\rho). \]
From this description, it is easy to calculate the value of $\mbox{E}$ for the generators of $\gG(A)^*$: 
\[ \mbox{E}(b_i^*) =  \mbox{A}_{i,s}^k \, \psi(b_k^*) \, \psi(a_s^*) + \psi(b_i^*) \, \psi(\alpha_i), \quad \mbox{E}(h^*) = - \mbox{S}(\alpha_i,h^*) \, \psi(b_i^*) \, \psi(a_i^*) + \rho(\bar{h}), \]
where $\bar{h} \in \gH$ is the element dual to $h^*$ via $\mbox{S}$. 

\medskip
\noindent
Now let $\alpha_0$ be a simple root in the subset of roots corresponding to the Lie group $K_J(A)$. Let  $h_0$ be the corresponding co-root.  
Notice that we may express the elements $[b_0,a_i]$, and $[b_0,h]$ in the form:
\[ [b_0,a_i] = \delta_{i,0} \mbox{A} \, h_0 + \mbox{B}_{i,0}^k \, a_k, \quad [b_0, h] = - \alpha_0(h) \, b_0, \quad \mbox{where} \quad \mbox{A} =  2/\mbox{S}(h_0,h_0).  \]

\noindent
Similarly, we may write $\mbox{ad}_{b_0}^*(b_i^*)$ and $\mbox{ad}_{b_0}^*(h^*)$ in the form:
\[  \mbox{ad}_{b_0}^*(b_i^*) = \delta_{i,0} \, \alpha_0 + \mbox{B}_{i,0}^k \, b_k^*, \quad  \mbox{ad}_{b_0}^*(h^*) =  - h^*(h_0) \, \mbox{A} \, a_0^*.  \]
Let us identify $\gB_-^* \hat{\otimes} \gB_-^*$, with $\gB_- \hat{\otimes} \gB_-^*$  by appling duality on the first factor. We express the quadratic homogeneous part in the expressions for $\mbox{E}(b_i^*)$ and $\mbox{E}(h^*)$ under this identification. These expressions transforms into the formal expressions in $\mbox{End}(\gB_-) = \gB_- \hat{\otimes} \gB_-^*$:
\[ \tilde{\mbox{E}}(b_i^*) = \mbox{A}_{i,s}^k \, a_k \otimes a_s^* + a_i \otimes \alpha_i, \quad \tilde{\mbox{E}}(h^*) = - \mbox{S}(\alpha_i,h^*) \, a_i \otimes a_i^*. \]
Notice that $\tilde{\mbox{E}}(b_i^*)$ represents the formal expression for $\mbox{ad}_{a_i}$ acting on $\gB_-$. Similarly, $\tilde{\mbox{E}}(h^*)$ represents $\mbox{ad}_{\bar{h}}$.
\noindent
Taking the commutator with $\sigma(b_0)$, and isolating the part in $\gB_-^* \hat{\otimes} \gB_-$:
\[ [\sigma(b_0), \tilde{\mbox{E}}(b_i^*)] = \tilde{\mbox{E}}([b_0,a_i]^*) + \delta_{i,0} \, \mbox{A} \, h_0 \otimes \alpha_0. \]
Reversing the identification and reducing the above equality to an equality of operators on $\hat{\mathbb{S}}(A)$, we see that the expression $[\sigma(b_0), \tilde{\mbox{E}}(b_i^*)]$ reduces to $[\sigma(b_0), \mbox{E}(b_i^*)]$, while the expression $\tilde{\mbox{E}}([b_0,a_i]^*)$ reduces to $\mbox{E}(\mbox{ad}_{b_0}^*(a_i^*)) - \rho(\bar{\alpha}_0)\, \mbox{Id}$. Finally, the term $\delta_{i,0} \, \mbox{A} \, h_0 \otimes \alpha_0$ reduces to the operator $\frac{1}{2} \delta_{i,0} \, \mbox{S}(\alpha_0,\alpha_0) \, \mbox{Id}$. On the other hand $\bar{\alpha}_0$ is $\frac{1}{2} \mbox{S}(\alpha_0,\alpha_0) \, h_0$, and so $\rho(\bar{\alpha}_0) = \frac{1}{2} \delta_{i,0} \, \mbox{S}(\alpha_0,\alpha_0)$. 

\medskip
\noindent
Putting this together, we have shown that $[\sigma(b_0), \mbox{E}(b_i^*)] = \mbox{E}(\mbox{ad}_{b_0}^*(b_i^*))$. To complete the lemma, one needs to consider the cases with $b_0$ replaced by $h_0, a_0$, and $b_i^*$ replaced by $h_i^*$ (the remaining cases follow on taking adjoints). These cases are much easier, and the computation is left to the reader.

\end{proof}

\noindent
{\bf A parametrized version of the cubic Dirac operator}:

\noindent

\noindent
To apply this to Dominant K-theory, we need a parametrized version of the above operator. We begin with a few preliminary observations: Let $\gG'(A)$ denoted the derived Lie algebra $\gG'(A) = [\gG(A), \gG(A)]$. There is a (split) short exact sequence of Lie algebras:
 \[ 0 \longrightarrow \gG'(A) \llra{i} \gG(A) \longrightarrow \mbox{t} \longrightarrow 0,  \]
 where $\mbox{t}$ is the quotient of the abelian Lie algebra $\gH$ by the subalgebra $\tilde{\gH}$ generated by the coroots: $ \tilde{\gH} = \oplus_i \C h_i$.  Let $\mbox{t}_{\R}$ denote the real form inside $\mbox{t}$. 
 
 \noindent
 Since we are working with a generalized Cartan matrix of compact type, the space $|\gC(A)|$ is equivalent to the Barycentric subdivision of the $n$-simplex $\mathcal{B}\Delta(n)$.  Define an affine map from $\mathcal{B} \Delta(n)$ to $\gG(A)^*$ by sending the Barycenter corresponding to a proper subset $J \subset I$ to the element $\rho_{I-J} \in \gH^*$, where we recall that for $K \subseteq I$, the elements $\rho_K$ are defined as:
\[ \rho_K = \sum_{i \in K} h_i^*. \]
This restricts to an affine map: $\mbox{f} : |\gC(A)| \rightarrow \gG(A)^*$. Since the stabilizer of the element $\rho_{I-J}$ is exactly the subgroup $K_J(A)$, we see that the map $\mbox{f}$ extends $K(A)$- equivariantly to a  proper embedding:
\[ \mbox{f} : X(A)   \longrightarrow \gG(A)^*.  \]
Furthermore, we may extend the map $\mbox{f}$ to an proper equivariant embedding:
\[ \mbox{F} : \mbox{t}_{\R}^* \times \R_+ \times X(A) \longrightarrow \gG(A)^*, \quad (t,r,x) \longmapsto t+r\, \mbox{f}(x),  \]
where $\R_+$ denotes the positive reals. Let $Y(A)$ denote the space $\mbox{t}_{\R}^* \times \R_+ \times X(A)$. Since $\tilde{\Kth}_{K(A)}^r(Y(A))$ is canonically isomorphic to $\tilde{\Kth}_{K(A)}^n(X(A))$, where $r$ is the rank of $K(A)$, we may as well give a geometric description of the former, which is slightly more convenient. 

\medskip
\noindent
Notice that the fundamental domain for the $K(A)$-action on $Y(A)$ is given by the subspace $\mbox{t}_{\R}^* \times \R_+ \times |\gC(A)|$. Indeed, the image of $\mbox{t}_{\R}^* \times \R_+ \times |\gC(A)|$ under the map $\mbox{F}$ belongs to the Weyl chamber $\mbox{C}$.

\medskip
\noindent
Now let $\mu + \rho \in \mbox{D}_+$ be dominant regular weight. We construct a Dominant Fredholm family over $Y(A)$, with the underlying vector space given by the product: $\mathbb{H}_{\mu} \times Y(A)$. Choose a continuous family $\mbox{B}_y$, of Hermitian inner-products on $\gG(A)$,  parametrized by points $y$ in the space  $\mbox{t}_{\R}^* \times \R_+ \times |\gC(A)|$. We require that $\mbox{B}_y$ is invariant under the stabilizer (in $K(A)$), of the point $y$. Since all stabilizers are compact Lie groups, it is easy to see that we may choose such a family, and that any two choices are equivariantly homotopic (see the Appendix for details). The $\omega$-conjugates $\mbox{S}_y = \mbox{B}^{\omega}_y$ form a continuous family of non-degenerate, symmetric, bilinear forms on $\gG(A)$. 

\medskip
\noindent
For a point $y$ in the fundamental domain, let $\lambda \in \mbox{C}$ be the element $\mbox{F}(y)$. Define the operator at $y$, to be $\Dirac_\lambda = \Dirac - \psi(\lambda)$, where $\psi(\lambda)$ denotes Clifford multiplication, with respect to the Clifford algebra generated under the relation $x^2 = \frac{1}{2}\mbox{S}_y(x,x)$. Here $\Dirac$ refers to the operator with symbol $\Dirac_1 + \Dirac_2$ defined earlier. By \ref{dirac}, we see that the operator $\Dirac$ is invariant under the stabilizer of $y$. Hence $\Dirac_\lambda$ is also fixed by stabilizer of the point $y$, and so we may therefore use the $K(A)$-action to extend the Fredholm family to all of $Y(A)$.  

\begin{claim}
Let $\mu \in \mbox{D}$ be a dominant weight. Then the family $(\mathbb{H}_{\mu}, \Dirac_{\lambda})$ is equivalent to the pullback of the character $\mu + \rho$ under the pinch map:
\[ \pi^* :  \mbox{DR}_T = \tilde{\Kth}^r_{K(A)}(K(A)_+\wedge_T S^r) \longrightarrow \tilde{\Kth}^r_{K(A)}(Y(A)). \] 
\end{claim}
\begin{proof}
Since the operators $\partial + \partial^{\dagger}$ and $\mbox{K}_{\lambda} = \mbox{K} - \psi(\lambda)$ commute, we may calculate the support of this family in two steps. Firstly, we calculate the kernel of $\partial + \partial^{\dagger}$ which is \cite{Ku}:
\[ \sum_{w \in W(A)} (-1)^w \, \mathbb{S}(\gH^*) \otimes \C_{w(\mu + \rho)} ,  \]
where the sign of a vector space indicates its parity. It is now clear that the operator $\mbox{K}_{\lambda}$ acting on this complex is supported at the $W$-orbit of $\mu + \rho$. Since $\mathbb{S}(\gH^*)$ represents the Thom class in ${}^A\tilde{\Kth}_T^r(S^r) = \tilde{\Kth}^r_{K(A)}(K(A)_+\wedge_T S^r)$, this operator is induced by the image of the character $\mu + \rho$.

\end{proof}

\begin{remark} The operator $\Dirac_{\lambda}$ is not a bounded operator. We rectify this as follows: 

\noindent
First notice that by construction, $\Dirac_{\lambda}$ commutes with the $T$-action, and therefore it has a weight space decomposition as a direct sum of operators acting on the individual finite-dimensional isotypical summands in $\mathbb{H}_{\mu}$. 

\noindent
Next, taking a regular element $h^* \in \gH_{\R}^*$, the equality $[\Dirac, \psi(h^*)] = \mbox{E}(h^*)$, and the formula for $\mbox{E}(h^*)$, implies that $\Dirac_{\lambda}$ has a discrete spectrum. Furthermore, one can explicitly compute $\Dirac_{\lambda}^2$ on the isotypical summands (see \cite{Ku} Section 3.4). It follows easily that $(1+\Dirac_{\lambda}^2)^{-1}$ is a well defined, continuously varying family of compact operators, in the norm topology. 

\noindent
Similarly, one can show that $\Dirac_{\lambda}(1+\Dirac_{\lambda}^2)^{-1/2}$ is a continuously varying family of bounded operators, in the compact open topology. It now follows, as in \cite{AS}, that $(\mathbb{H}_\mu, \Dirac_{\lambda})$ may be suitable normalized to define a continuous family of operators with values in $\mathcal{F}(\mathcal{H})$. 
\end{remark}

\section{Dominant K-theory for the extended compact type}
\noindent
Thus far, we have restricted our attention to Kac-Moody groups of compact type because this assumption greatly simplified the Dominant K-theory of $X(A)$. Let us now explore the behaviour of Dominant K-theory for generalized Cartan matrices that are one level up in terms of complication:

\noindent
Recall that a generalized Cartan matrix $A$ is of extended compact type if there exists a decomposition $I = I_0 \coprod J_0$, such that $(a_{i,j})_{j,k \in J}$ is of non-finite type if and only if $I_0 \subseteq J$.
Before we begin the analysis of the extended compact type, let us prove a simple claim about the Davis complex, which will help unify the argument. We begin with a definition:

\begin{defn} \label{ext} Let $W$ be an arbitrary Coxeter group on a generating set $I$. Let $J,K \subseteq I$ be any subsets. Define the set of pure double coset representatives: ${}^J \overline{W} {}^K$:
\[ {}^J\overline{W}{}^K = \{ w \in {}^J W {}^K \, | \, \, w\, W_K(A) \, w^{-1} \, \cap W_J(A) = 0 \}. \]
Similarly, define the set of (right) maximally pure elements ${}^J \overline{W}{}^K_s$:
\[ {}^J\overline{W}{}^K_s = \{ w \in {}^J \overline{W}{}^K \, | \, \, w \notin {}^J\overline{W}{}^L \,\, \mbox{for all proper inclusions} \,\, K \subset L \}. \]

\end{defn}

\begin{claim} Assume that $A$ is a generalized Cartan matrix of extended compact type as defined above. Assume that $|I_0| = n+1$, for some integer $n >1$. Let $K \subseteq I$ be any subset, then the groups : $\mbox{H}^p_c(\Sigma/W_K(A),\Z)$ are trivial if $p \neq \{0,n\}$, and  
\[ \mbox{H}_c^n(\Sigma/W_K(A),\Z) = \Z [ {}^K \overline{W} {}^{I_0}_s].  \]
Moreover, $\mbox{H}^0_c(\Sigma/W_K(A),\Z)$ is nonzero if and only if $\Sigma/W_K(A)$ is compact and contractible, in particular for this case ${}^K \overline{W} {}^{I_0}_s$ is empty. 
\end{claim}
\begin{proof}
The proof is essentially the argument given in \cite{D3} (8.3.1,8.3.3). First notice that $\Sigma/W_K(A)$ may be identified with a subspace $\mbox{A}_K \subseteq \Sigma$ given by the "K-Sector":
\[ \mbox{A}_K = \bigcup_{w \in {}^K W} w\, |\gC(A)| \, \, \subseteq \Sigma. \]
This identification provides a retraction of $\Sigma/W_K(A)$ from $\Sigma$. Let us now order elements of ${}^K W$ in a sequence $\{ w_1, w_2, \ldots, \}$ such that $l(w_k) \leq l(w_{k+1})$. Consider an exhaustive filtration of $\mbox{A}_K$ by subspaces $\mbox{U}_k$ given by the union of the first $k$ tranlates $w_i \, |\gC(A)|, \, i \leq k$. And let $\overline{\mbox{U}}_k$ denote its complement. Consider the following long exact sequence:
\begin{equation} \label{ex}
\ldots \rightarrow \mbox{H}^*(\mbox{A}_K, \overline{\mbox{U}}_{k-1}) \rightarrow \mbox{H}^*(\mbox{A}_K, \overline{\mbox{U}}_k) \rightarrow \mbox{H}^*(\overline{\mbox{U}}_{k-1}, \overline{\mbox{U}}_k) \rightarrow \ldots 
\end{equation}
We prove by induction on $k$, that $\mbox{H}^*(\mbox{A}_K,\overline{\mbox{U}}_k)$ has cohomology only in degrees $\{0,n\}$, then we shall identify the generators in degree $n$. 

\medskip
\noindent
Setting $\overline{\mbox{U}}_0 = \mbox{A}_K$, notice \cite{D3}(8.3.1), that we may  identify the groups $\mbox{H}^*(\overline{\mbox{U}}_{k-1}, \overline{\mbox{U}}_k)$ with $\mbox{H}^*(|\gC(A)|, |\gC^k(A)|) $, where $|\gC^k(A)|$ denotes the realization of the nerve of the subcategory:
\[  \gC^k(A) = \{ J \in \gC(A) \, | \, l(w_k r_j) > l(w_k) \, \, \mbox{and} \, \, w_k r_j \in {}^K W \, \, \mbox{for some} \, \,  j \in J  \}. \]
Assume by induction that the groups $\mbox{H}^*(\mbox{A}_K,\overline{\mbox{U}}_{k-1})$ are concentrated in degrees $\{0,n\}$. By the definition of extended compact type, it is easy to see that $\mbox{H}^*(|\gC(A)|, |\gC^k(A)|)$ is nontrivial if and only if $\gC^k(A)$ is empty, or is the full subcategory: $\{ J \in \gC(A) | J \cap I_0 \neq \varnothing \}$. 

\medskip
\noindent
 In the case when $\gC^k(A)$ is empty, the long exact sequence \ref{ex}, along with the induction assumption, forces $\mbox{H}^0(\mbox{A}_K,\overline{\mbox{U}}_k)$ to have a free summand, which implies that $\overline{\mbox{U}}_k = \varnothing$ i.e. $\mbox{U}_k = \mbox{A}_K$ or that $\mbox{A}_K$ is compact. Since $\mbox{A}_K$ is contractible (being a retract of $\Sigma$), we have shown in this case that $\mbox{H}^*_c(\Sigma/W_K(A), \Z) = \Z$, concentrated in degree zero.  
 
 \medskip
 \noindent
 For $\gC^k(A)$ being subsets intersecting $I_0$ nontrivially, it is easy to see that $\mbox{H}^p(|\gC(A)|, |\gC^k(A)|)$ is nontrivial only if $p=n$, and is isomorphic to $\Z$ in that degree. By induction, the long exact sequence \ref{ex} splits and we pick up a free generator. 
 
 \medskip
 \noindent
 It is straightforward to check that the latter case is equivalent to the condition: $w_k \in {}^K \overline{W} {}^{I_0}_s$. Hence we get one generator in degree $n$ for each maximally pure element. 
\end{proof}

\medskip
\noindent
We may now attempt to calculate $\Kth_{K(A)}^*(X(A))$ for groups $K(A)$ of extended compact type along the same lines as the compact type using the spectral sequence for the skeletal decomposition of $X(A)$. We stick with the assumption that $|I_0| = n+1$, for $n>1$. 

\medskip
\noindent
To claculate $\Kth_{K(A)}^*(X(A))$, we begin with the $E_2$ term of the spectral sequence:
\[ E_2^{p,*} = \varprojlim  {}^p \Kth_{K(A)}^*(K(A)/K_{\bullet}(A)) = \varprojlim  {}^p \mbox{DR}_{\bullet}[\beta^{\pm1}] \Rightarrow \Kth_{K(A)}^p(X(A))[\beta^{\pm1}],  \]
where we recall that $\varprojlim  {}^p \mbox{DR}_{\bullet}$ is canonically isomorphic to the equivariant cohomology of the Davis complex $\Sigma$ (\ref{DC}), with values in the ring $\mbox{DR}_T$: 
 \[ \varprojlim {}^p \mbox{DR}_{\bullet} = \mbox{H}^p_{W(A)}(\Sigma, \mbox{DR}_T). \] 
 Next, we decompose the functor $\mbox{DR}_{\bullet}$ as $\bigoplus \mbox{DR}_{\bullet}^K$ indexed by $K \subseteq I$:
\[ \varprojlim {}^p \mbox{DR}_{\bullet} = \bigoplus \varprojlim {}^p \mbox{DR}_{\bullet}^K = \bigoplus \mbox{H}^p_{W(A)}(\Sigma, \Z[W(A)/W_K(A)]) \otimes \mbox{R}_T^K. \] 
It is easy to see that $\mbox{H}^p_{W(A)}(\Sigma, \Z[W(A)/W_K(A)]) = \mbox{H}_c^p(\Sigma/W_K(A), \Z)$. Hence using the previous claim we get that $\varprojlim {}^i \mbox{DR}_{\bullet} = 0$ if  $ i \neq \{0,n\}$, and in those degrees, we have:
\begin{equation} \label{eq2} \varprojlim {}^0 \mbox{DR}_{\bullet} =   \bigoplus_{|{}^K W|<\infty} \mbox{R}_T^K, \quad \quad \varprojlim {}^n \mbox{DR}_{\bullet} =  \bigoplus_{K \subset I} \Z[{}^K \overline{W}{}^{I_0}_s] \otimes \mbox{R}_T^K. 
\end{equation} 
The finiteness of ${}^{K} W$ implies (using the previous claim) that the set ${}^K \overline{W} {}^{I_0}_s$ is empty, and hence the summand corresponding to $K \subset I$ is absent in $\varprojlim {}^n\mbox{DR}_{\bullet}$. 

\medskip
\noindent
As before, all classes in degree $n$ are permanent cycles. We proceed to show that the classes in degree $0$ are permanent cycles as well. Our method of proof is interesting in its own right since it demonstrates the interesting relationship between the Dominant K-theory of $X(A)$ and $X(A_0)$, where $A_0$ is the generalized Cartan matrix of compact type given by $(a_{i,j})_{i,j \in I_0}$. 

\medskip
\noindent
Consider the action of $K(A_0)$ on $X(A_0)$. Since $K_{I_0}(A)$ is an extension of $K(A_0)$ by a subtorus, it is easy to see that the $K(A_0)$-action on $X(A_0)$ extends to $K_{I_0}(A)$. Indeed, we have a homeomorphism between each orbit $K(A_0)/K_J(A_0) = K_{I_0}(A)/K_J(A)$ for $J \subseteq I_0$. It is not hard to see that $X(A)$ is the classifying space of proper $K_{I_0}(A)$-actions, though we shall not need that fact here. 

\medskip
\noindent
Consider the proper $K(A)$-space: $K(A) \times_{K_{I_0}(A)} X(A_0)$. By the universality of $X(A)$, we have a unique map:
\[ \varphi: K(A) \times_{K_{I_0}(A)} X(A_0) \longrightarrow X(A). \]
Recall that $X(A_0)$ can be seen as a homotopy colimit of a diagram defined on the category of all proper subsets of $I_0$. It is obvious that $\varphi$ is the map induced by a natural transformation given the the identity map on each orbit $K(A)/K_J(A)$ for $J \subset I_0$. This induces a map:
\[ \varphi^*: \Kth^*_{K(A)}(X(A)) \longrightarrow {}^A \Kth^*_{K_{I_0}(A)}(X(A_0)), \]
where we bear in mind that the right hand side is the restriction of the Dominant K-theory with respect to the group $K(A)$, to the subgroup $K_{I_0}(A)$. Setting up the spectral sequence to compute ${}^A \Kth_{K_{I_0}(A)}^*(X(A_0))$, as in Section 5 gives us:
\[ E_2^{p,*} = \mbox{H}_{W_{I_0}(A)}^p(\Sigma_0, \mbox{DR}_T) [\beta^{\pm1}] \Rightarrow {}^A \Kth_{K_{I_0}(A)}^p(X(A_0))[\beta^{\pm1}],  \]
where $\Sigma_0$, is the Davis complex for $W_{I_0}(A)$, (alternatively $\Sigma_0 = X(A_0)^T$). The coefficient ring $\mbox{DR}_T$ is a free $\Z$ module on a disjoint union of transitive $W(A)$-sets. More precisely, we have a decomposition indexed over subsets $K \subseteq I$:
\[ \mbox{DR}_T = \bigoplus \mbox{DR}_T^K, \quad \mbox{where} \quad \mbox{DR}_T^K \cong \Z[W(A)/W_K(A)] \otimes \mbox{R}_T^K. \]
Under the action of $W_{I_0}(A)$, the set $W(A)/W_K(A)$ decomposes further as:
\[ W(A)/W_K(A) = \coprod_{w \in {}^{I_0} W {}^K} W_{I_0}(A)/W_{K_w}(A), \quad \mbox{where} \quad W_{K_w}(A) = W_{I_0}(A) \cap \, w \, W_K(A) \, w^{-1}.  \]
Repeating the argument given for Theorem \ref{thm1}, we start with the fact that:
\[ \mbox{H}^p_{W_{I_0}(A)}(\Sigma_0, \Z[W_{I_0}(A)/W_{K_w}(A)]) = \mbox{H}^p_c(\Sigma_0/W_{K_w}(A), \Z) \]
 and that these groups are nontrivial in at most one degree. In addition, we have:
 \[ \mbox{H}^0_c(\Sigma_0/W_{K_w}(A),\Z)  =\Z \, \, \, \mbox{if} \, \, \,  K_w = I_0 \, \, \, \mbox{and} \, \, \, \mbox{H}^n_c(\Sigma_0/W_{K_w}(A),\Z) = \Z \, \, \mbox{if} \, \, K_w = \varnothing, \]
 and that $\mbox{H}^p_c(\Sigma_0/W_{K_w}(A),\Z) = 0$ for all $p$ if $K_w  \neq \{\varnothing, I_0\}$. 
 
 \medskip
 \noindent
Now, for the condition, $K_w = I_0$ to hold for some $w \in {}^{I_0} W {}^K$, the group $W_{K_w}(A)$ must be infinite, and hence we need only to consider the case $I_0 \subseteq K $. Let $B^K$ denote the set of elements $w \in {}^{I_0} W ^K$, such that $K_w = I_0$. We therefore have:
\[ \mbox{H}^0_{W_{I_0}(A)}(\Sigma_0, \mbox{DR}_T) = \bigoplus_{I_0 \subseteq K} \Z[B^K] \otimes \mbox{R}_T^{K}.  \]
Now given $w \in B^K$, any character in $ w\, \mbox{R}_T^K$ is $W_{I_0}(A)$ invariant, and hence the above elements represent elements of Dominant K-theory given by $K_{I_0}(A)$-invariant, rank one vector bundles over $X(A_0)$. In particular, they are permanent cycles. 

\medskip
\noindent
Next consider the elements that contribute to degree $n$. It follows from the definition \ref{ext} that the condition: $K_w = \varnothing$ for $w \in {}^{I_0} W {}^K$, is equivalent to $w^{-1} \in {}^K \overline{W} {}^{I_0}$. Hence we have:
\[   \mbox{H}^n_{W_{I_0}(A)}(\Sigma_0, \mbox{DR}_T) =  \bigoplus_{K \subset I} \Z[{}^K \overline{W}{}^{I_0}] \otimes \mbox{R}_T^K. \] 
Comparing these computations with those in \ref{eq2}, we see easily that the map of spectral sequences is injective on the $E_2$-stage. Invoking functoriality, we conclude that:
\begin{thm} \label{excpct} Let $K(A)$ be a Kac-Moody group of extended compact type and let $|I_0| = n+1$, for any integer $n >1 $. Then we have an isomorphism of graded groups:
\[ \Kth_{K(A)}^*(X(A)) = \bigoplus_{|{}^K W|<\infty} \mbox{R}_T^K [\beta^{\pm 1}] \, \, \bigoplus_{K \subset I} \Z[{}^K \overline{W} {}^{I_0}_s] \otimes \mbox{R}_T^K [\beta^{\pm 1}]. \]
The classes $\Z[{}^K \overline{W} {}^{I_0}_s] \otimes \mbox{R}_T^K $ appear naturally in degree $n$, and the others in degree $0$. Furthermore, the following map is injective:
\[ \varphi^*: \Kth^*_{K(A)}(X(A)) \longrightarrow {}^A \Kth^*_{K_{I_0}(A)}(X(A_0)), \quad \mbox{with}\]
\[ {}^A \Kth_{K_{I_0}(A)}^*(X(A_0)) = \bigoplus_{I_0 \subseteq K} \Z[B^K] \otimes \mbox{R}_T^K[ \beta^{\pm 1}] \, \, \bigoplus_{K \subset I} \Z[{}^K \overline{W} {}^{I_0}] \otimes \mbox{R}_T^K [\beta^{\pm 1}]. \]
where $B^K$ denotes the set of elements $w \in {}^{I_0} W ^K$, such that $W_{I_0} (A) \subseteq w \, W_K(A) \, w^{-1}$. 
\end{thm}

\begin{remark}  \label{rmk1} The classes $\bigoplus_{K \subset I} \Z[ {}^K \overline{W} {}^{I_0}] \otimes \mbox{R}_T^K$ have a simple description as the ideal generated by characters corresponding to the $W(A)$-orbit of  the dominant weights, $\mbox{D}$, that belong to the interior of the Weyl chamber of $K_{I_0}(A)$. The classes $\bigoplus_{K \subset I} \Z[ {}^K \overline{W} {}^{I_0}_s] \otimes \mbox{R}_T^K$ may then be identified with those characters of the form described above, which also belong to the antidominant Weyl chamber of $K_{J_0}(A)$. 

\noindent
Hence the direct sum decompositions given above can be seen as a partition of  ideals in $\mbox{DR}_T$ given by characters corresponding to weights that belongs to the interior of the Weyl chamber of $K_{I_0}(A)$. 
\end{remark}
\section{Some related remarks}
\noindent
{\bf The Affine Case}.
 
 \noindent
 
 \noindent
We would like to make some brief comments relating our result with results of Freed-Hopkins-Teleman. For simplicity, we deal only with the untwisted Affine case. The verification of details in this section are left to the interested reader. 

\medskip
\noindent
Let $K(A)$ denote the Affine Kac-Moody group corresponding to the group of loops on a simply connected, simple compact Lie group $\mbox{G}$. Let $\tilde{\mbox{L}}\mbox{G}$ denote the central extension of the group of polynomial loops on $\mbox{G}$, and let $\mathbb{T}$ denote the rotational circle acting by reparametrization of loops. Then we may identify $K(A)$ with the group $\mathbb{T} \ltimes \tilde{\mbox{L}}\mbox{G}$. 

\noindent
Let $\mbox{T}$ denote the maximal torus of $\mbox{G}$, and let $\mbox{K}$ denote the central circle of $K(A)$. The maximal torus $T$ of $K(A)$ may be decomposed as $T = \mbox{T} \times \mathbb{T} \times \mbox{K}$. Let $\mbox{k}$ and $\mbox{d}$ denote the weights dual to $\mbox{K}$ and $\mathbb{T}$ respectively, under the above decompositon. The nonzero vectors in the Tits cone $\mbox{Y}$ are given by the upper-half plane consisting of vectors that have a positive $\mbox{k}$-coefficient. Hence, the Dominant representation ring of $T$ may be identified as:
\[ \mbox{DR}_T = \Z[e^{\pm \mbox{d}}] + e^{\mbox{k}} \, \mbox{R}_{\mbox{T}} [ e^{\mbox{k}}, e^{\pm \mbox{d}} ], \]
where $\mbox{R}_{\mbox{T}}$ denotes the representation ring of $\mbox{T}$. The central weight $\mbox{k}$ is known as "level", and we may decompose our maximal dominant Hilbert space $\mathcal{H}$ into a (completed) sum of eigen-spaces corresponding to the level: 
\[ \mathcal{H} = \hat{\bigoplus}_{m\geq 0} \mathcal{H}_m. \]
Now assume that $X$ is a finite, proper $K(A)$-CW complex on which the center $\mbox{K}$ acts trivially. Hence the action of $K(A)$ on $X$ factors through the central quotient: $\mathbb{T} \ltimes \mbox{LG}$. One has a corresponding decomposition of Dominant K-theory:
\[ \Kth_{K(A)}^*(X) = \bigoplus_{m \geq 0} {}^m\Kth_{K(A)}^*(X), \]
where ${}^m\Kth_{K(A)}^0(X)$ is defined as homotopy classes of $\mathbb{T}\ltimes \mbox{LG}$ equivariant maps from $X$ to the space $\mathcal{F}(\mathcal{H}_m)$, of Fredholm operators on $\mathcal{H}_m$. There is a similar decomposition for homology. We now proceed to relate this decomposition to twisted K-theory. 

\noindent
Recall that $X$ is a proper $K(A)$-CW complex, hence all its isotropy subgroups of $X$ are compact. It is well known that the subgroup of pointed polynomial loops $\Omega\mbox{G} \subset \mathbb{T} \ltimes \mbox{LG}$ contains no compact subgroup, and hence $\Omega \mbox{G}$ acts freely on $X$. It follows that ${}^m\Kth_{K(A)}^0(X)$ may be seen as homotopy classes of $\mathbb{T} \ltimes \mbox{LG}$- equivariant sections of a fiber bundle over $X/\Omega \mbox{G}$, with fiber $\mathcal{F}(\mathcal{H}_m)$, and structure group $\mathbb{PU}(\mathcal{H}_m)$ (with the compact open topology). This is one possible definition of the energy equivariant version of twisted equivariant K-theory.

\medskip
\noindent
Recall that projective Hilbert bundles of the form above are classifed by a "twist" belonging to $\mbox{H}_{\mathbb{T} \ltimes LG}^3(X, \Z)$ \cite{AS}. For the space $X$ above, this twist is simply the pullback of the universal twist $[m]$ in $\mbox{H}^3_{\mathbb{T} \ltimes LG}(X(A) , \Z) = \Z$, where $X(A)$ is the classifying space of proper $K(A)$-actions. We therefore conclude that the Dominant K-theory groups ${}^m\Kth_{K(A)}^*(X)$ represent the (energy equivariant) $m$-twisted equivariant K-theory groups of the space $X/\Omega \mbox{G}$.

\medskip
\noindent
It is easy to see that the space $X(A)$ is homeomorphic to the affine space of principal $\mbox{G}$-connections for the trivial bundle over $\mbox{S}^1$ \cite{KM}. It follows that $X(A)/\Omega \mbox{G}$ is homeomorphic to $\mbox{G}$, via the holonomy map. The geometric isomorphism we constructed in Section 5 identified the irreducible representation $\mbox{L}_\mu$, with highest weight $\mu$, with the character $e^{\mu + \rho}$ considered as an element in  $\Kth_{K(A)}^n(X(A))$, (here $n$ may be identified with the rank of the compact group $\mbox{G}$). The Weyl element $\rho$ has level given by the the dual Coxeter number, $h^{\vee}$, and so we recover the energy equivariant version of the theorem of Freed-Hopkins-Teleman \cite{FHT} (Section 15), that relates the $k+h^{\vee}$-twisted equivariant K-theory groups, with highest weight (or positive energy) level $k$ representations of the Loop group.

\medskip
\noindent
 {\bf Real forms of Kac-Moody groups}.
 
 \noindent
 
 \noindent
 So far we have been considering the unitary form $K(A) \subset G(A)$, where $G(A)$ denotes the complex points of the Kac-Moody group. There is a functorial construction of the Kac-Moody group over an arbitrary field \cite{T}. In particular, we may consider the real form $G^{\R}(A)$. One may identify the real form as the fixed subgroup: $G^{\R}(A) = G(A)^{\tau}$, where $\tau$ denotes complex conjugation. This action preserves the subgroups $G_J(A)$, we may define the real form of the parabolic subgroups: $G^{\R}_J(A) = G_J(A)^{\tau} = G_J(A) \cap G^{\R}(A)$. Notice also that the action by complex conjugation descends to the homogeneous spaces: $G(A)/G_J(A)$. 
 
 \medskip
 \noindent
Let $K^{\R}(A)$ denote the "orthogonal form" inside $G^{\R}(A)$. Hence, $K^{\R}(A)$ is given by the fixed points of the involution $\omega$ acting on $G^{\R}(A)$. In other words, the orthogonal form is given by: $K^{\R}(A) = G^{\R}(A)^{\omega} =  K(A) \cap G^{\R}(A)$. And along the same lines, define subgroups of $K^{\R}(A)$ by $K^{\R}_J(A) := G_J^{\R}(A)^{\omega} = K_J(A) \cap G^{\R}(A)$.

 \begin{claim}
The following canonical inclusions are homeomorphisms:
 \[ K^{\R}(A)/K^{\R}_J(A) \subseteq G^{\R}(A)/G^{\R}_J(A) \subseteq \{ G(A)/G_J(A) \}^{\tau}. \]
 \end{claim}
 \begin{proof}
 We will prove these equalities by comparing the Bruhat decomposition over the fields $\R$ and $\C$. First recall the Bruhat decompositon for the space $G(A)/G_J(A)$:
 \begin{equation} \label{eq1} G(A)/G_J(A) = \coprod_{w \in W^J} B\, \tilde{w} \, B/B, \end{equation}
where $\tilde{w}$ denotes any element of $N(T)$ lifting $w \in W(A)$. In fact, the general theory of Kac-Moody groups allows us to pick the elements $\tilde{w}$ to belong to $G^{\R}(A)$. Moreover, the subspace $B\, \tilde{w} \, B/B$ has the structure of the complex affine space: $\C^{l(w)}$. In addition, the set of right coset representatives: $\C^{l(w)}$ may be chosen to belong to $K(A)$. 

\medskip
\noindent
Working over $\R$ we have the Bruhat decomposition:
\begin{equation} \label{eq3} G^{\R}(A)/G^{\R}_J(A) = \coprod_{w \in W^J} B^{\R}\, \tilde{w} \, B^{\R}/B^{\R}, \end{equation}
where the subspace $B^{\R}\, \tilde{w} \, B^{\R}/B^{\R}$ has the structure of the real affine space $\R^{l(w)}$, with the representatives $\R^{l(w)}$ belong to $K^{\R}(A)$. By comparing equations \ref{eq1} and \ref{eq3} it follows that 
\[ \R^{l(w)}  = B^{\R} \, \tilde{w} \, B^{\R}/B^{\R} = \{ B\, \tilde{w} \, B/B \} ^{\tau},  \]
which assembles to the statement of the theorem as $w$ ranges over $W^J$.  
\end{proof}
 \begin{defn} Define the real (finite-type) Topological Tits building $X^{\R}(A)$ as the $K^{\R}(A)$-space:
 \[ X^{\R}(A) = \frac{K^{\R}(A)/T^{\R} \times |\gC(A)|}{\sim},  \]
where we identify $(gT^{\R},x)$ with $(hT^{\R},y)$ iff $x=y \in \Delta_J(n)$, and $g=h \, \mbox{mod} \, K^{\R}_J(A)$.  We note that the group $T^{\R} = K^{\R}_{\varnothing}(A)$, is the subgroup of two torsion points in $T$. 
\end{defn}
\noindent
As before, we may identify this space as the homotopy colimit of a suitable functor on the category $\gC(A)$. We may also relate $X^{\R}(A)$ to the space $X(A)$ studied earlier via:
\begin{thm} $X^{\R}(A)$ is equivariantly homeomorphic to the fixed point space $X(A)^{\tau}$. Moreover, it is a finite, proper, $K^{\R}(A)$- CW space which is $K^{\R}_J(A)$-equivariantly contractible for $J \in \gC(A)$. 
 \end{thm}
 \begin{proof}
 The space $X^{\R}(A)$ is clearly a finite, proper, $K^{\R}(A)$-CW complex. Also, the previous claim shows that $X^{\R}(A)$ is equivariantly homeomorphic to $X(A)^{\tau}$. It remains to show that it is $K^{\R}_J(A)$-equivariantly contractible for all $J \in \gC(A)$. This follows immediately once we notice that the proof of contractibility given in theorem \ref{Tits}, is $\tau$ equivariant. 
 \end{proof}
 \begin{remark} The space $X^{\R}(A)$ is in fact equivalent to the classifying space of proper $K^{\R}(A)$-actions: it sufficies to show that any compact subgroup of $K^{\R}(A)$ is conjugate to a subgroup of  $K^{\R}_J(A)$ for some $J \in \gC(A)$. This follows from the general properties of Buildings \cite{D4}. We thank P.E. Caprace for pointing this out to us. 
 \end{remark}
 
 \noindent
 {\bf The group $\mbox{E}_{10}$}.
 
 \noindent
 
 \noindent
There has been recent interest in the real form of the Kac-Moody group $\mbox{E}_{10}$. This group is of extended compact type, with $J_0$ being the unique extended node so that $(a_{i,j})_{i,j\in I_0}$ is the affine type generalized Cartan matrix for $\mbox{E}_9$.

\noindent
 It appears that the real form of $\mbox{E}_{10}$ encodes the symmetries of the dynamics ensuing from 11-dimensional supergravity \cite{DHN, DN} \footnote{We thank Hisham Sati for explaining this to us}. The dynamics may be (conjecturally) expressed in terms of null geodesics in the Lorentzian space $G^{\R}(\mbox{E}_{10})/K^{\R}(\mbox{E}_{10})$. The group $\mbox{E}_{10}$ also appears in a different fashion when 11-dimensional supergravity is compactified on a 10-dimensional torus,  \cite{J, BGH}. There is a conjectural description of the latter in terms of "billiards" in the Weyl chamber of $\mbox{E}_{10}$ \cite{DHN, DN, BGH}. 
 
 \medskip
 \noindent
 Since $X^{\R}(\mbox{E}_{10})$ is the universal proper $K^{\R}(\mbox{E}_{10})$-space, it seems reasonable to guess that the configuration space of the dynamics indicated above (perhaps away from singularities), may be related to the space $X^{\R}(\mbox{E}_{10})$. 

\medskip
\noindent
The affine type that was just discussed is relevant in field theory. The Dominant K-theory of the affine space $X(A)$ can be seen as the receptacle for the D-brane charge in the $G/G$-gauged WZW model.  It can also be identified with the Verlinde algebra of a modular functor that corresponds to the categorification of Chern-Simons theory. 

\medskip
\noindent
Invoking theorem \ref{excpct} and remark \ref{rmk1} for the group $\mbox{E}_{10}$, we have:
\begin{thm} Let $A = \mbox{E}_{10}$ denote the Kac-Moody group as above, then the group $K_{I_0}(\mbox{E}_{10})$ is the (energy extended) Affine group $\mbox{E}_9$, and the following restriction map, is an injection:
 \[ \varphi^*: \tilde{\Kth}^*_{E_{10}}(X(\mbox{E}_{10})) \longrightarrow {}^{A} \tilde{\Kth}^*_{E_9}(X(\mbox{E}_9)). \]
where reduced Dominant K-theory denote the kernel along any orbit $T$-orbit. Moreover, the maximal dominant  $\mbox{E}_{10}$-representation restricts to a dominant $\mbox{E}_9$-representation, inducing a map:  
 \[  \mbox{St}: \, {}^A \tilde{\Kth}^*_{E_9}(X(\mbox{E}_9)) \longrightarrow \tilde{\Kth}^*_{E_9}(X(\mbox{E}_9)). \]
The map $\mbox{St}$ is injective, and has image given by the ideal generated by characters that correspond to weights in the interior of the (energy extended) Weyl chamber of $\mbox{E}_9$, which are in the Weyl-orbit of dominant characters of $\mbox{E}_{10}$. 

\noindent
The image of $\tilde{\Kth}^*_{E_{10}}(X(\mbox{E}_{10}))$ under $\mbox{St}$ is given by those characters of the form described above that also belong to the antidominant Weyl chamber of $K_{J_0}(\mbox{E}_{10})$.
 \end{thm}
\begin{remark}
Consider the action map : $K(\mbox{E}_{10}) \times_{K^{\R}(E_{10})} X^{\R}(\mbox{E}_{10}) \longrightarrow X(\mbox{E}_{10})$. This induces:
\[ \Kth^*_{K(E_{10})} (X(\mbox{E}_{10})) \longrightarrow {}^{A} \Kth^*_{K^{\R}(E_{10})} (X^{\R}(\mbox{E}_{10})), \]
which allows us to construct elements in ${}^A \tilde{\Kth}_{K^{\R}(E_{10})} (X^{\R}(\mbox{E}_{10}))$ via the previous theorem. The behaviour of this map is still unclear to us. 
\end{remark}

\newpage
\section{Dominant K-homology}
\noindent
To give a meaningful definition of the Dominant K-homology groups, we need the definition of the equivariant dual of a finite, proper $K(A)$-CW spectrum. Our definition is motivated by \cite{Kl}:
\begin{defn} \label{dual} For a finite, proper $K(A)$-CW spectrum $X$, define the equivariant dual of $X$ as the homotopy fixed point spectrum:
\[ DX = \Map(X, K(A)_+)^{hK(A)},  \]
where $K(A)_+$ denotes the suspension spectrum: $\Sigma^{\infty} K(A)_+$ with a right $K(A)$-action. 
\end{defn}

\noindent
For a proper $K(A)$-CW spectrum $X$, we show below that $DX$ has the equivariant homotopy type of a finite, proper $K(A)$-CW spectrum via the residual left $K(A)$-action on $K(A)_+$. For example, let $X$ be a single proper orbit $K(A)_+\wedge_{G}S^0$ for some compact Lie subgroup $G \subseteq K(A)$. In this case, we will show below that $DX$ is equivalent to the proper $K(A)$-spectrum $K(A)_+\wedge_{G}S^{\gG}$, where $S^{\gG}$ denotes the one point compactification of the adjoint representation of $G$ on its Lie algebra. 
\begin{defn} \label{homology} The Dominant K-homology groups of a finite, proper $K(A)$-CW spectrum $X$ are defined as:
\[ \Kth^{K(A)}_k(X) = \Kth_{K(A)}^{-k}(DX).  \]
We may extend this definition to arbitrary proper $K(A)$-CW complexes by taking direct limits over sub skeleta. 
\end{defn}

\begin{claim} Let $G \subseteq K(A)$ be a compact Lie subgroup, let $S^{\gG}$ denote the one point compactification of the adjoint representation of $G$ on its Lie algebra. Then the equivariant dual of a proper orbit: $X = K(A)_+\wedge_G S^0$ is given by $DX \simeq K(A)_+ \wedge_G S^{\gG}$.
\end{claim}
\begin{proof} 
Consider the sequence of equalities:
\[ DX = \Map(K(A)_+\wedge_G S^0, K(A)_+)^{hK(A)} =  \Map(S^0, K(A)_+)^{hG}  = K(A)_+^{hG}. \]
Now recall \cite{Kl}, that there is a $G \times G$-equivalence of spectra: $\Map(G_+, S^{\gG}) \simeq G_+ $. Since $G_+$ is a compact Lie group, and $K(A)$ is a free $G$-space, we get an $K(A) \times G$-equivalence:
\[ \Map(G_+, K(A)_+\wedge_G S^{\gG}) = K(A)_+\wedge_G \Map(G_+,S^{\gG}) \simeq K(A)_+. \]
Taking homotopy fixed points, we get the required $K(A)$-equivalence:
\[ DX = K(A)_+^{hG} \simeq \Map(G_+, K(A)_+\wedge_G S^{\gG})^{hG} \simeq K(A)_+\wedge_G S^{\gG}, \]
which completes the proof. 
\end{proof} 

 \noindent
Now recall that $X(A)$ was a finite homotopy colimit of proper homogeneous spaces of the form $K(A)/K_J(A)$. Therefore, we may write the equivariant dual of $X(A)$ as the following homotopy inverse limit in the category of proper $K(A)$-CW spectra:
\[ DX(A) = {\rm holim}_{J \in \gC(A)} \; K(A)_+ \wedge_{K_J(A)} S^{\gG(J)},  \]
where $\gG(J)$ denotes the Lie algebra of $K_J(A)$. This is a finite $K(A)$-CW spectrum, and its Dominant K-cohomology is defined as the Dominant K-homology of $X(A)$. We have:
 \begin{thm} Let $K(A)$ be a Kac-Moody group of compact type. Assume that $K(A)$ is not of finite type, and let $r$ be the rank of $K(A)$. Define the reduced Dominant K-homology $\tilde{\Kth}^{K(A)}_*(X(A))$, to be the kernel of the pinch map $\pi$. Then we have an isomorphism of graded groups:
\[ \tilde{\Kth}^{K(A)}_*(X(A)) = \mbox{R}_T^{\varnothing}[\beta^{\pm1}], \]
where $\mbox{R}_T^{\varnothing}$ is graded so as to belong entirely in degree $-r$, and $\beta$ is the Bott class in degree 2. Moreover, the identification of $\tilde{\Kth}^{K(A)}_{-r}(X(A))$ with $\mbox{R}_T^{\varnothing}$ is induced by the inclusion of any $T$-orbit.

\end{thm}
\begin{proof}
The description of the equivariant dual of $X(A)$ as a homotopy inverse limit, yields a spectral sequence:
\[ E_2^{p,*} = \varinjlim  {}^p \Kth_{K(A)}^*(K(A)_+\wedge_{K_{\bullet}(A)}S^{\gG(\bullet)}) = \varinjlim  {}^p \Kth_{K_{\bullet}(A)}^*(S^{\gG(\bullet)})\Rightarrow \Kth^{K(A)}_{p-*}(X(A)). \]
To simplify the notation, we will use the Thom isomorphism theorem \ref{Thom} to identify the (covariant) functor $J \mapsto {}^A \Kth_{K_J(A)}^r(S^{\gG(J)})$ with $\mbox{DR}_{\bullet}$. Hence, we may write our spectral sequence as:
\[ E_2^{p,*} = \varinjlim  {}^p \mbox{DR}_{\bullet}[\beta^{\pm}] \Rightarrow \Kth^{K(A)}_{p-r}(X(A))[\beta^{\pm1}]. \]
The strategy in proving the above theorem is not new.  We decompose the functor $\mbox{DR}_{\bullet}$ into summands $\tilde{\mbox{D}}\mbox{R}_{\bullet}^K$ indexed over $K \subseteq I$, and then consider two separate cases:

\medskip
\noindent
Recall the decomposition of the ring $\mbox{DR}_T$ indexed by subsets $K \subseteq I$: $\mbox{DR}_T = \bigoplus \mbox{DR}_T^K$. This induces a decomposition of the functor $\mbox{DR}_{\bullet}$ using the push forward map \ref{PF}: 
\[ \mbox{DR}_{\bullet} = \bigoplus \tilde{\mbox{D}}\mbox{R}_{\bullet}^K, \quad \mbox{where} \quad \tilde{\mbox{D}}\mbox{R}_J^K = {\iota_J}(\mbox{DR}_T^K). \]
We now proceed as before by considering two cases:

\medskip
\noindent
{\bf The case K=I}.
It is clear that the functor $\tilde{\mbox{D}}\mbox{R}_{\bullet}^I$ takes the value zero on all nonempty proper subsets $J$, and that $\tilde{\mbox{D}}\mbox{R}_{\varnothing}^I = \mbox{R}_T^I$. Hence we get 
\[ \varinjlim {}^p \tilde{\mbox{D}}\mbox{R}_{\bullet}^I =  0, \, \,  \mbox{if} \, \, p\neq n, \, \,  \mbox{and}   = \mbox{R}_T^I \, \, \mbox{if} \, \, p=n. \]
Since $\mbox{R}_T^I$ represents the highest filtration in the spectral sequence, it admits a map to $\mbox{DR}_T$ via $\pi$. At the end of this argument, we shall show that the elements $\mbox{R}_T^I$ are permanent cycles representing elements in $\Kth^{K(A)}_{n-r}(X(A))$. 

\medskip
\noindent
{\bf The case K $\subset$ I}.
To calculate higher derived colimits of the remaining summands: $\tilde{\mbox{D}}\mbox{R}_{\bullet}^K$, we begin by first spliting the functor $\tilde{\mbox{D}}\mbox{R}_{\bullet}^K$ out of another functor $\mbox{E}_{\bullet}^K$, where $K \subset I$ is any proper subset. The functor $\mbox{E}_{\bullet}^K$ is defined by:
\[   J \longmapsto \mbox{E}_J^K = \mbox{R}_T^K \otimes \tilde{\Z} \otimes_{\Z[W_K(A)]} \Z[W(A)/W_J(A)],  \]
where $\tilde{\Z}$ stands for the sign representation of $W_K(A)$. One has a map from the functor $\tilde{\mbox{D}}\mbox{R}_{\bullet}^K$ to $\mbox{E}_{\bullet}^K$ defined as follows: 

\noindent
Let $\mbox{L}_{\mu}$ be an irreducible representation in $\tilde{\mbox{D}}\mbox{R}_J^K$ with highest weight $\mu$. Assume that we have the equality: $\mu + \rho_J = w^{-1}(\tau)$, for some $\tau \in \mbox{D}_K$, and $w \in W(A)$. Define a map 
\[ \tilde{\mbox{D}}\mbox{R}_J^K \longrightarrow \mbox{E}_J^K, \quad \quad \mbox{L}_{\mu} \longmapsto (-1)^we^{\tau} \otimes 1 \otimes  w,  \]
where $(-1)^w$ denotes the sign of the element $w$, and $e^{\tau}$ stands for a character in $\mbox{R}_T^K$. It is easy to see that this map well defined and functorial. Moreover, one may check that the retraction is given by:
\[ \mbox{E}_J^K \longrightarrow \tilde{\mbox{D}}\mbox{R}_J^K, \quad \quad e^{\tau} \otimes 1 \otimes w \longmapsto (-1)^w \iota_J (w^{-1}e^{\tau}). \]
Hence, $\varinjlim {}^p \tilde{\mbox{D}}\mbox{R}_{\bullet}^K$ splits from the homology of the chain complex $\mbox{R}_T^K \otimes \tilde{\Z}\otimes_{\Z[W_K(A)]} \mbox{C}_*$, where $\mbox{C}_*$ is the simplicial chain complex of the Davis complex $\Sigma$. Since $W_K(A)$ is a finite subgroup, this simplicial complex is $W_K(A)$-equivariantly contractible. Thus $\mbox{C}_*$ is $W_K(A)$-equivariantly equivalent to the constant complex $\Z$ in 
dimension zero. 

\medskip
\noindent
It follows that $\varinjlim {}^p \tilde{\mbox{D}}\mbox{R}_{\bullet}^K = 0$ if $p>0$, and $ \varinjlim  \tilde{\mbox{D}}\mbox{R}_{\bullet}^K$ splits from the group $\mbox{R}_T^K \otimes \tilde{\Z} \otimes_{\Z[W_K(A)]} \Z$. Furthermore, if $K$ is nonempty, then notice that the group $\mbox{R}_T^K \otimes \tilde{\Z} \otimes_{\Z[W_K(A)]} \Z$ is two torsion, hence 
the map to $\tilde{\mbox{D}}\mbox{R}_{\bullet}^K$ is trivial. This shows that $\varinjlim \tilde{\mbox{D}}\mbox{R}_{\bullet}^K = 0$ if $K$ is not the empty set. Finally, for $K = \varnothing$, one simply observes that the above retraction is an isomorphism, and hence $\varinjlim \tilde{\mbox{D}}\mbox{R}_{\bullet}^{\varnothing} = \mbox{R}_T^{\varnothing}$.

\medskip
\noindent
To complete the proof, it remains to show that all elements in $\mbox{R}_T^I$ discussed in the first case are permanent cycles. Unfortunately, we cannot invoke a geometric argument as before. Instead we consider the $K(A)$-equivariant action map:
\[ \varphi: K(A) \times_{N(T)} \Sigma \longrightarrow X(A),  \]
where $N(T)$ stands for the normalizer of $T$ in $K(A)$. To show that $\mbox{R}_T^I$ consists of permanent cycles, it is sufficent to show that it belongs to the image of $\varphi_*$:
\[ \varphi_* : {}^A \Kth^{N(T)}_{n-r}(\Sigma) = \Kth^{K(A)}_{n-r}(K(A) \times_{N(T)} \Sigma) \longrightarrow \Kth^{K(A)}_{n-r}(X(A)). \]
Consider now the $N(T)$-equivariant pinch map $\pi: \Sigma \longrightarrow W(A)_+ \wedge S^n$.
Taking the $N(T)$-equivariant dual of $\pi$ yields $D\pi : W(A)_+\wedge S^{r-n} \longrightarrow D\Sigma$
where $r$ is the rank of the maximal torus $T$. We have a description of $D\Sigma$ as a homotopy limit in the category of proper $N(T)$-CW spectra: 
\[ D\Sigma = {\rm holim}_{J \in \gC(A)} \; W(A)_+ \wedge_{W_J(A)} S^{\gH_{\R}}. \]
Taking $N(T)$ orbits of $D\Sigma$ yields a non-equivariant CW spectrum: 
\[ S^0 \wedge_{N(T)} D\Sigma = {\rm holim}_{J \in \gC(A)} \; S^0 \wedge_{W_J(A)} S^{\gH_{\R}}.  \]
Notice that the orbit space $\gH_{\R}/W_J(A)$ can be identified with a cone within $\gH_{\R}$, and hence, the spectra $S^0 \wedge_{W_J(A)} S^{\gH_{\R}}$ are contractible for all $J \neq \varnothing$. It is now easy to see that the dual pinch map induces an equivalence of non-equivariant spectra: 
\[ D\pi : S^{r-n} = S^0 \wedge_{N(T)} W(A)_+ \wedge S^{r-n}  \longrightarrow S^0 \wedge_{N(T)}  D\Sigma. \] 
Now given any weight $\mu \in \mbox{D}_I$, let $\mbox{L}_{\mu}$ be the one dimension highest weight representation corresponding to $\mu$. These represent canonical elements of ${}^A \Kth_{T}^{r-n}(S^{r-n})$ which extend to corresponding elements in ${}^A \Kth_{N(T)}^{r-n}(D\Sigma)$ using the above equivalence.

\noindent
This shows that elements of $\mbox{R}_T^I$ lift to elements in ${}^A \Kth^{N(T)}_{n-r}(\Sigma)$, thus completing the proof. 

\end{proof}

\newpage
\appendix
\section{ A compatible family of metrics}
\noindent
Given an indecomposable, symmetrizable generalized Cartan matrix $A$, let us fix a non-degenerate, invariant, bilinear form $\mbox{B}$ on the Lie algebra $\gG(A)$, \cite{Ku}. This form, however, is not positive definite in general. In this section we will explore the ways in which one can perturb $\mbox{B}$ into positive definite bilinear forms in an equivariantly compatible way (to be made precise below). 

\medskip
\noindent
Let $J \in \gC(A)$ be any subset, and let $K_J(A)$ be the corresponding compact subgroup of $K(A)$. We may decompose the Lie algebra $\gG(A)$ as:
\[ \gG(A) = \gZ_J \oplus \gH_J \oplus \eta_- \oplus \eta_+,  \]
where $\gH_J$ is the subspace of co-roots generated by $h_j$ for $j \in J$, and $\gZ_J$ is the centralizer of $K_J(A)$ given by:
\[ \gZ_J = \{ h \in \gH \, | \, \alpha_j(h) = 0, \, \forall \, j \in J \}. \]
It is easy to see that $\mbox{B}$ restricts to a positive form on the $K_J(A)$-invariant subspace $\gG_J$:
\[ \gG_J = \gH_J \oplus \eta_- \oplus \eta_+. \]
Notice also that the action of $K_J(A)$ on $\gZ_J$ is trivial, and that $\gZ_J$ is orthogonal to $\gG_J$. Let $\mbox{B}(J)$ denote the restriction of $\mbox{B}$ to $\gG_J$. From the above observation, it is clear that extending $\mbox{B}(J)$ to a $K_J(A)$-invariant Hermitian inner-product on $\gG(A)$, is equivalent to constructing a (non-equivariant) Hermitian inner-product on $\gZ_J$. Let us therefore define the space of extensions:

\begin{defn} Let $\mbox{\bf Met}(J)$ denote the space of $K_J(A)$-invariant Hermitian inner-products on $\gG(A)$, extending $\mbox{B}(J)$. Hence, $\mbox{\bf Met}(J)$ is equivalent to the space of Hermitian inner-products on $\gZ_J$, and in particular, it is contractible. 
\end{defn}
\noindent
Notice that there is an obvious restriction map from $\mbox{\bf Met}(K)$ to $\mbox{\bf Met}(J)$ if $J \subseteq K$. Thus we obtain an inverse system of spaces, $\mbox{\bf Met}(\bullet)$, on the poset $\gC(A)$.

\medskip
\noindent
Let $\mbox{\bf Met}$ be the homotopy inverse limit of this system \cite{BK}. In simple terms, $\mbox{\bf Met}$ is nothing other than the space of Hermitian inner-products on $\gG(A)$, which are equivariantly parametrized over the building $X(A)$, and such that the metric over any point in the fundamental domain in $X(A)$ belongs to some space $\mbox{\bf Met}(J)$. 

\noindent
Since $\mbox{\bf Met}(J)$ is contractible for all $J$, it is easy to see using a spectral sequence argument for homotopy inverse limits, that $\mbox{\bf Met}$ is weakly contractible. In particular, it is non-empty. Therefore we conclude:

\begin{claim}
The topological space consisting of families of Hermitian inner-products on $\gG(A)$, equivariantly parametrized over $X(A)$ and extending $\mbox{B}$, is non-empty. 
\end{claim}

\end{document}